\documentclass[11pt,a4paper]{article}
\usepackage[T1]{fontenc}
\usepackage[utf8]{inputenc}
\usepackage{lmodern}
\usepackage[margin=29mm]{geometry}
\usepackage{amsmath,amssymb,amsthm,mathtools}
\usepackage{microtype}
\usepackage{enumitem}
\usepackage{xcolor}
\usepackage{hyperref}
\hypersetup{colorlinks=true,linkcolor=black,citecolor=blue!45!black,
urlcolor=blue!45!black,pdftitle={Finite Time Singularity of the Kahler-Ricci Flow on compact Kahler surfaces are of Type I},
pdfauthor={Yeyun Xu and Linfeng Zhou}}
\numberwithin{equation}{section}
\newtheorem{theorem}{Theorem}[section]
\newtheorem{proposition}[theorem]{Proposition}
\newtheorem{lemma}[theorem]{Lemma}

\theoremstyle{definition}

\theoremstyle{remark}
\newtheorem{remark}[theorem]{Remark}
\newcommand{\R}{\mathbb R}
\newcommand{\C}{\mathbb C}
\newcommand{\Z}{\mathbb Z}
\newcommand{\PP}{\mathbb P}
\newcommand{\Ric}{\operatorname{Ric}}
\newcommand{\Rm}{\operatorname{Rm}}
\newcommand{\Vol}{\operatorname{Vol}}
\newcommand{\diam}{\operatorname{diam}}

\newcommand{\Bl}{\operatorname{Bl}}
\newcommand{\reg}{\mathrm{reg}}

\DeclareMathOperator{\supp}{supp}

\DeclareMathOperator{\PD}{PD}
\setlist{itemsep=3pt,topsep=5pt}
\title{Finite Time Singularities of the K\"ahler--Ricci Flow\\
on Compact K\"ahler Surfaces are of Type~I}
\author{Yeyun Xu \and Linfeng Zhou}
\date{}

\begin{document}
\maketitle
\begin{NoHyper}
\begingroup
\renewcommand{\thefootnote}{}
\footnotetext{Supported by Fundamental and Interdisciplinary Disciplines
Breakthrough Plan of the Ministry of Education of China(JYB2025XDXM112).}
\endgroup
\end{NoHyper}
\vspace{-1.5em}
\begin{abstract}
We establish a Type~I curvature bound for finite-time volume-collapsing
K\"ahler--Ricci flows on compact K\"ahler surfaces. No symmetry assumption
is imposed, and the initial K\"ahler class need not be rational.
The proof combines local symplectic topology with the classification
of gradient K\"ahler--Ricci shrinking solitons \cite{LW}.
Bamler's compactness and structure theory
\cite{BamEntropy,BamCompact,BamStructure} provides the shrinking limits
used in the argument.
\end{abstract}

\section{Introduction}

\noindent\hspace*{2em}Finite-time singularities of the K\"ahler--Ricci flow
on compact complex surfaces have been understood in several geometric
regimes. For extinction in an anticanonical class, the soliton existence
results of Tian \cite{TianDelPezzo} and Wang--Zhu \cite{WZ}, together with
normalized-flow convergence \cite[Corollaries~1.5 and~4.3]{DS}, give a
Type~I curvature bound. In the non-collapsing case,
$\liminf_{t\nearrow T}\Vol(X,g(t))>0$, Conlon--Hallgren--Ma
\cite[Theorem~A]{CHM} proved the global Type~I estimate without a
symmetry assumption. For volume collapse, Tosatti--Zhang
\cite[Proposition~2.5 and Corollary~2.6]{TZcollapse} reduce the
non-extinction case to a Fano fibration over a curve. Xu--Zhang
\cite{XZ} obtain a global Type~I estimate when the fibration has no
singular fibres, and estimates away from the singular fibres in general.

Our main result gives the volume-collapsing estimate by a local
symplectic argument. Together with \cite[Theorem~A]{CHM}, it implies
that every finite-time singularity of the K\"ahler--Ricci flow on a
compact K\"ahler surface is of Type~I.
\begin{theorem}\label{thm:main}
Let $(X,J,\omega_0)$ be a connected compact K\"ahler surface, and let
$\omega(t)$ be the maximal solution of
\begin{equation}\label{eq:krf}
 \partial_t\omega(t)=-\Ric(\omega(t)),\qquad
 \omega(0)=\omega_0,\qquad 0\le t<T<\infty.
\end{equation}
Suppose that
\[
 \lim_{t\nearrow T}\Vol(X,\omega(t))
 =\lim_{t\nearrow T}\frac12\int_X\omega(t)^2=0.
\]
If $g(t)$ is the Riemannian metric associated with $\omega(t)$, then
there is a constant $C<\infty$, depending only on the initial flow,
such that
\begin{equation}\label{eq:typeI}
 \sup_X|\Rm(g(t))|_{g(t)}\le\frac{C}{T-t},
 \qquad 0\le t<T.
\end{equation}
In particular, the singularity at the maximal time $T$ is of Type~I.
\end{theorem}

\begin{remark}\label{rem:higher-dimensions}
The conclusion fails in higher complex dimensions. Miao--Tian
\cite[Theorem~1.1]{MT} exhibit smooth Fano threefolds in Mori--Mukai
family No.~2.23 whose normalized K\"ahler--Ricci flows $\widehat g(s)$
have unbounded curvature. In the anticanonical normalization,
\[
 \begin{aligned}
 g(t)&=(1-t)\widehat g(-\log(1-t)),\qquad 0\le t<1,\\
 \sup_{0\le t<1}(1-t)\|\Rm(g(t))\|_\infty
 &=\sup_{s\ge0}\|\Rm(\widehat g(s))\|_\infty=\infty.
 \end{aligned}
\]
The unnormalized flow has maximal time $1$ and
$\Vol(X,g(t))=(1-t)^3\Vol(X,g(0))\to0$, so these are finite-time
volume-collapsing Type~II singularities.
\end{remark}

The Type~I conclusion of Theorem~\ref{thm:main} also appears in
Xu--Zhang \cite{XZTypeI} and
Cifarelli--Conlon--Hallgren--Zhang \cite{CCHZ}, but the methods differ.
Xu--Zhang establish smoothness of tangent shrinkers through their Fano
fibration structure and the topology of singular-fibre neighbourhoods;
their global estimate then uses secondary tangent flows, local
cohomology, and curvature control from the Bamler--Cifarelli--Conlon--Deruelle
(BCCD) model \cite[Sections~3--4]{XZTypeI}.
Cifarelli--Conlon--Hallgren--Zhang reduce possible orbifold singularities
to Wahl singularities and exclude them by comparing intersection-matrix
traces; an entropy gap between the smooth tangent models and nontrivial
flat quotient cones then rules out Type~II blow-up
\cite[Sections~4.5--4.6 and~5]{CCHZ}.

Our proof uses a symplectic obstruction inside a single fibre
neighbourhood. Its intersection form has negative index at most one.
Replacing a hypothetical filling of a nontrivial spherical quotient
by its minimal resolution, we compare the exceptional configuration
with a retained rational ruling and obtain incompatible Chern-number
identities. This excludes flat quotient annuli in the original flow
and hence orbifold points in smaller-scale shrinking limits. At scales
$r_i^2\ll T-t_i$, the cohomological evolution forces the limiting
K\"ahler form to be exact. The smooth shrinker classification then
leaves only the Gaussian model, contradicting the strictly negative
entropy of the extracted limit.

The symplectic ingredients are the rational rulings and blow-up
description in \cite[Section~4]{Markov}, the structure of genus-zero
curves in nef classes \cite[Theorem~1.5]{LZnef}, and the contact and
resolution theory of quotient singularities
\cite[Section~3.A]{CNPP}, \cite[Section~2.2]{CP}.
Bamler's entropy estimates, compactness, and structure theorems
\cite{BamEntropy,BamCompact,BamStructure} provide the shrinking limits;
the two-stage extraction follows \cite[Lemma~2.11]{BCCD}.
Li--Wang's classification \cite[Theorem~1.1]{LW}, which requires no
a priori curvature bound, supplies the final list of smooth models.

The paper is organized as follows.
Section~\ref{sec:geometry} records the fibration geometry and local
topology. Section~\ref{sec:filling} proves the spherical quotient
filling obstruction. Section~\ref{sec:limits} applies it to the
K\"ahler--Ricci flow and establishes exactness of smaller-scale
limits. Section~\ref{sec:tangents} constructs the shrinking limits
and proves Theorem~\ref{thm:main}.

\paragraph{AI Usage Statement:}
This work arose while the author was thinking about the evolution of the
K\"ahler--Ricci flow on $\operatorname{Bl}_{p_1,p_2}\mathbb{CP}^2$.
During the preparation of this work, the author made substantial use of
ChatGPT as a research assistant. In particular, it was used to compute
relevant concepts in algebraic geometry, to help the author learn,
understand, and use algebraic topology and Bamler's theory, and to draft
and polish portions of the manuscript. The direction and strategy of the
proof are the author's own. The author verified all AI-assisted output and
takes full responsibility for the correctness of this publication.

\paragraph{Acknowledgment.}
The first author thanks the geometric analysts
who participated in the summer school at Nanjing University and
the participants in the Differential Geometry Group Seminar in ECNU for
helpful discussions. The second  author would like to acknowledge the support of the JYB2025XDX112 fund.

\section{Geometry of the collapsing fibration}\label{sec:geometry}

We use the flow and maximal time in Theorem~\ref{thm:main}, with
$[\Ric(\omega)]=2\pi c_1(X)$, and set
\begin{equation}\label{eq:class}
 \tau=T-t,\qquad
 \alpha=[\omega_0]-2\pi T c_1(X),\qquad
 [\omega(t)]=\alpha+2\pi\tau c_1(X)
 \quad\text{in }H^2(X;\R).
\end{equation}
The class $\alpha$ is nef, and volume collapse is equivalent to
$\alpha^2=0$. In this section, assume
$\alpha^2=0$ and $\alpha\ne0$.

\begin{proposition}[Fibration reduction]\label{prop:reduction}
Suppose that $\alpha^2=0$ and $\alpha\ne0$. There is a holomorphic
surjection with connected fibres
\[
 f:X\longrightarrow\Sigma
\]
onto a smooth compact curve, with generic fibre $\PP^1$, and a K\"ahler
form $\beta$ on $\Sigma$ such that $\alpha=f^*[\beta]$.
Moreover, for a constant $c>0$ independent of $t$,
\begin{equation}\label{eq:schwarz}
 \omega(t)\ge c f^*\beta,\qquad 0\le t<T.
\end{equation}
Consequently $f:(X,g(t))\to(\Sigma,g_\beta)$ is uniformly
$c^{-1/2}$-Lipschitz.
\end{proposition}

\begin{proof}
The existence of a Fano fibration and the equality of limiting classes
follow from \cite[Proposition~2.5 and Corollary~2.6]{TZcollapse}.
These statements apply to real N\'eron--Severi classes; in dimension two
the argument uses $H^{2,0}(X)=0$ and imposes no rationality condition on
$[\omega_0]$. Their base is normal and has dimension strictly less than
two. Since the pullback class is nonzero, the base has dimension one.
A normal complex algebraic curve is smooth
\cite[Tag~0B8Y]{Stacks}. The generic fibre is a smooth connected Fano
curve, hence $\PP^1$.

The global lower bound \eqref{eq:schwarz} is
\cite[Lemma~2.2]{XZ}; it is the lower-bound part of the Schwarz estimate
and holds also on the singular fibres. Since $f$ is holomorphic,
\eqref{eq:schwarz} gives $|df(v)|_{g_\beta}^2\le c^{-1}|v|_{g(t)}^2$.
Integrating along paths proves the Lipschitz assertion.
\end{proof}

\begin{lemma}\label{lem:fibres}
Under the assumptions of Proposition~\ref{prop:reduction}, every
singular fibre is reduced and has the form $E\cup C$, where
\begin{equation}\label{eq:twocurves}
 E^2=C^2=-1,\qquad E\cdot C=1.
\end{equation}
Both components are smooth rational curves. Near their intersection
there are holomorphic coordinates in which $f(u,v)=uv$.
There are no multiple fibres.
\end{lemma}

\begin{proof}
The two-component conclusion is \cite[Lemma~2.1]{XZ}. The following
argument also records the multiplicities needed for the local topology.
A smooth generic fibre $F$ has $F^2=0$ and
$c_1(X)\cdot F=2$. Write a reducible fibre as
$F_q=\sum_jm_jD_j$, with $D_j$ irreducible and reduced and $m_j>0$.
Connectedness and $F_q\cdot D_j=0$ imply $D_j^2<0$. As $D_j$ is vertical,
\[
 c_1(X)\cdot D_j=\frac{[\omega_0]\cdot D_j}{2\pi T}>0.
\]
Adjunction for the arithmetic genus gives
\[
 c_1(X)\cdot D_j=2-2p_a(D_j)+D_j^2.
\]
It follows that $p_a(D_j)=0$ and $D_j^2=-1$. The normalization formula
for arithmetic genus then shows that $D_j$ is smooth and rational.
The intersection equations become
\[
 m_j=\sum_{\ell\ne j}m_\ell(D_j\cdot D_\ell).
\]
Set $d_\ell=\sum_{j\ne\ell}D_j\cdot D_\ell$. Summing the preceding
equalities gives $\sum_\ell m_\ell(d_\ell-1)=0$. Each $d_\ell$ is an
integer at least one, so each is one. The connected graph therefore has
two vertices and a single edge of intersection multiplicity one.
The two fibre multiplicities are equal, and
$2=c_1(X)\cdot F_q$ forces both to be one.

If instead $F_q=mD$ is irreducible, then $D^2=0$ and $c_1(X)\cdot D>0$.
Adjunction now gives $p_a(D)=0$ and $c_1(X)\cdot D=2$, so $m=1$.
This also rules out an irreducible singular fibre.

Two smooth curves with intersection number one meet at a unique
transverse point. In coordinates adapted to the components, a base
coordinate is $uvh(u,v)$ with $h$ a holomorphic unit. Replacing one
coordinate by its product with $h$ gives $f=uv$.
\end{proof}

Let $D\Subset\Sigma$ be an open disc whose closure contains at most one
singular value and whose boundary contains none. Write
\begin{equation}\label{eq:localU}
 U=f^{-1}(D),\qquad
 \delta=\#\{\text{singular values in }D\}\in\{0,1\}.
\end{equation}

\begin{lemma}\label{lem:localtop}
The open manifold $U$ is simply connected and has one end. Its
intersection form on compact integral two-cycles is
\begin{equation}\label{eq:localQ}
 \begin{cases}
 H_2(U;\Z)=\Z[F],\quad F^2=0,&\delta=0,\\[2pt]
 H_2(U;\Z)=\Z[E]\oplus\Z[C],\quad
 Q_U=\begin{pmatrix}-1&1\\1&-1\end{pmatrix},&\delta=1.
 \end{cases}
\end{equation}
In particular $Q_U$ is negative semidefinite with negative index
$\delta$, and $\alpha|_U=0$ in real cohomology.
If the total number of singular fibres is $k$, then
$b_2^+(X)=1$ and $b_2^-(X)=1+k$.
\end{lemma}

\begin{proof}
Blow down one component of each singular fibre. These $(-1)$-curves
are pairwise disjoint, so the contractions can be performed independently.
The map $f$ descends through each contraction. The other component
becomes a smooth reduced rational fibre of square zero; the local model
$uv$ shows that the descended map is a submersion there. The resulting
proper submersion is therefore a smooth $S^2$-bundle over $\Sigma$.
Over the contractible disc $D$ it is trivial, and hence
\[
 U\cong S^2\times D\quad(\delta=0),\qquad
 U\cong\Bl_p(S^2\times D)\quad(\delta=1)
\]
as oriented smooth manifolds, where $p$ is an interior point.
These manifolds are simply connected and have one end; a blow-up changes
neither property because it replaces an interior ball by the disc bundle
of Euler number $-1$. In the blown-up case the basis $(F,E)$ has form
$\operatorname{diag}(0,-1)$; the change $C=F-E$ gives
\eqref{eq:localQ}. A sphere bundle over a closed oriented surface has
$b_2^+=b_2^-=1$, and each blow-up adds one negative direction.
Finally $\beta|_D$ is exact, so $f^*[\beta]|_U=0$.
\end{proof}

\section{Spherical quotient fillings}
\label{sec:filling}

All intersection forms in this section use the symplectic orientation.
For a compact oriented four-manifold $V$ with rational homology sphere
boundary, the natural map
\[
 \iota_V:H^2(V,\partial V;\R)\longrightarrow H^2(V;\R)
\]
is an isomorphism. For $c\in H^2(V;\R)$, its square means
\[
 c^2=\langle\iota_V^{-1}c\smile c,[V,\partial V]\rangle.
\]
For symplectic $V$, write $c_{1,\R}(TV)\in H^2(V;\R)$ for the real
image of its first Chern class $c_1(TV)\in H^2(V;\Z)$ and set
$\kappa_V=c_{1,\R}(TV)^2$. Integral curve classes are mapped to real
homology when used in intersection forms on real vector spaces.

A boundary is \emph{strongly convex} if a neighbourhood admits a primitive
$\lambda$ of the symplectic form whose Liouville vector field points
outward.  Contact identifications below preserve coorientation.
If $\Gamma\subset U(2)$ is finite and acts freely on $S^3$, denote by
$\xi_{\mathrm{std}}$ the quotient of the standard contact structure on
$S^3$.

\begin{theorem}\label{thm:obstruction}
Let $(M,\Omega)$ be a connected closed symplectic four-manifold with
\[
b_2^+(M)=1,\qquad b_2^-(M)=1+\delta,
\qquad \delta\in\{0,1\}.
\]
Suppose that $M$ contains an embedded symplectic sphere $F$ with $F^2=0$.
There is no connected compact symplectic domain $W\subset M\setminus F$
such that $H_1(W;\R)=0$ and its connected strongly convex boundary is
contactomorphic to $(S^3/\Gamma,\xi_{\mathrm{std}})$ for a nontrivial
finite subgroup $\Gamma\subset U(2)$ acting freely on $S^3$.
\end{theorem}

We give the contact and resolution constructions before proving the
theorem.  This also fixes the orientation and Chern-class conventions in
the replacement argument.

\begin{lemma}\label{lem:resolutionfill}
Let $\Gamma\subset U(2)$ be nontrivial and act freely on $S^3$.
The minimal resolution of $\C^2/\Gamma$ contains a compact strongly
convex K\"ahler domain $P$ whose boundary is contactomorphic to
$(S^3/\Gamma,\xi_{\mathrm{std}})$.
The domain $P$ retracts onto the exceptional divisor
\[
\mathcal D=D_1\cup\cdots\cup D_r,\qquad
D_j\cong\PP^1,\qquad D_j^2=-a_j,\quad a_j\ge2,
\]
which is a nonempty connected normal-crossing tree with negative
definite intersection matrix.  Consequently
\[
b_1(P)=b_3(P)=0,\quad b_2(P)=r,\quad
\chi(P)=1+r,\quad \sigma(P)=-r.
\]
\end{lemma}

\begin{proof}
Write $V=\C^2/\Gamma$, let $o$ be the image of the origin, and let
$\pi:Y\to V$ be the minimal resolution. Since the action is
free off the origin, $\pi$ is an isomorphism over $V\setminus\{o\}$.
The exceptional divisor $\mathcal D=\pi^{-1}(o)_{\mathrm{red}}$ is a
nonempty connected simple normal-crossing tree of smooth rational
curves with $D_j^2=-a_j\le-2$ \cite[Section~2.2]{CP}; its intersection
matrix is negative definite \cite[Lemma~(14.1)]{Lipman}.

\smallskip\noindent
\emph{The sublevel set and its topology.}
Choose positive-degree homogeneous generators $p_1,\ldots,p_N$ of
$\C[z_1,z_2]^\Gamma$. They give a closed embedding $V\hookrightarrow\C^N$.
Writing $\bar p_\nu$ for the descended functions, set
\[
 \rho=\sum_\nu|\bar p_\nu|^2,\qquad
 f=\rho\circ\pi=\sum_\nu|\bar p_\nu\circ\pi|^2,\qquad
 d^c=\tfrac{i}{2}(\bar\partial-\partial).
\]
The function $f$ is smooth and proper, $f^{-1}(0)=\mathcal D$, and
$dd^cf\ge0$, with strict positivity off $\mathcal D$. Indeed, the
quotient embedding is an immersion on $V\setminus\{o\}$, where $\pi$
is biholomorphic. On the covering space, the radial vector field $R$
satisfies
\[
 R\!\left(\sum_\nu|p_\nu|^2\right)
 =2\sum_\nu(\deg p_\nu)|p_\nu|^2>0.
\]
Hence $df\ne0$ on $\{f>0\}$. Fix $c>0$ and take the compact
smooth domain
\[
 P=P_c:=\{f\le c\},\qquad \partial P=\{f=c\},\qquad
 \mathcal D\Subset\operatorname{int}P.
\]

We construct the deformation inside this fixed domain $P$.
The resolution is projective over $V$
\cite[Corollary~(27.2)]{Lipman}, so there is a closed holomorphic embedding
\[
 J_0:Y\hookrightarrow V\times\PP^m\hookrightarrow\C^N\times\PP^m,
 \qquad
 \eta=J_0^*(\operatorname{pr}_1^*\omega_{\mathrm E}
              +\operatorname{pr}_2^*\omega_{\mathrm{FS}}).
\]
The form $\eta$ is K\"ahler. Use its metric $g_\eta$ for distances,
exponential maps and gradients in the following construction.
Use the smooth embedding
\[
 \jmath:\PP^m\hookrightarrow\operatorname{Herm}(m+1),\qquad
 \jmath([z])=\frac{zz^*}{z^*z},
\]
so that $(\mathrm{id}_{\C^N}\times\jmath)\circ J_0$ embeds $Y$ in a
real Euclidean space. The compact normal-crossing divisor $\mathcal D$
is locally contractible. Applying
\cite[Appendix, Theorem~A.7]{Hatcher} to its image and restricting
the resulting neighbourhood retraction to $Y$ gives an open
neighbourhood $U_0\subset Y$ and a continuous retraction
\[
 \mathfrak r:U_0\longrightarrow\mathcal D,\qquad
 \mathfrak r|_{\mathcal D}=\mathrm{id}_{\mathcal D}.
\]

Write $N_b(\mathcal D)=\{x\in Y:d_\eta(x,\mathcal D)<b\}$.
By compactness of $\mathcal D$, choose $b>0$ so that
\[
 \overline{N_{3b}(\mathcal D)}\subset\operatorname{int}P,\qquad
 \operatorname{inj}_{g_\eta}(x)>2b
 \quad(x\in N_b(\mathcal D)).
\]
Continuity of $\mathfrak r$ makes
\[
 U=\{x\in U_0\cap N_b(\mathcal D):
             d_\eta(x,\mathfrak r(x))<b\}
\]
an open neighbourhood of $\mathcal D$. The short-geodesic homotopy
\[
 K_t(x)=\exp_x\!\bigl(t\exp_x^{-1}(\mathfrak r(x))\bigr),
 \qquad (x,t)\in U\times[0,1],
\]
is continuous and satisfies
\[
 \begin{aligned}
 d_\eta(K_t(x),\mathcal D)
 &\le d_\eta(x,\mathcal D)+t\,d_\eta(x,\mathfrak r(x))<2b,\\
 K_0&=\mathrm{id}_U,\qquad K_1=\mathfrak r|_U,\qquad
 K_t|_{\mathcal D}=\mathrm{id}_{\mathcal D}.
 \end{aligned}
\]
Thus $K:U\times[0,1]\to\operatorname{int}P$ is a homotopy relative
to $\mathcal D$, with its entire image contained in $P$.

Since $P\setminus U$ is compact and disjoint from $f^{-1}(0)$,
choose $0<a<c$ with $P_a:=\{f\le a\}\subset U$. The vector field
$X=-\nabla f/|\nabla f|^2$ is smooth on $\{f>0\}$ and satisfies
$df(X)=-1$. Its flow $\Phi_s$ obeys
\[
 f(\Phi_s(x))=f(x)-s,
 \qquad 0\le s\le f(x)-a,\quad a\le f(x)\le c;
\]
existence throughout this interval follows because the trajectories
stay in the compact band $\{a\le f\le c\}$. Define
\[
 H_t(x)=
 \begin{cases}
 x,&f(x)\le a,\\
 \Phi_{t(f(x)-a)}(x),&f(x)>a,
 \end{cases}
 \qquad 0\le t\le1,
\]
where the two expressions agree continuously at $f=a$ since the
flow time tends to zero. Then $H_t(P)\subset P$, $H_t|_{P_a}=\mathrm{id}$,
and $H_1(P)=P_a$. Consequently
\[
 G_t(x)=
 \begin{cases}
 H_{2t}(x),&0\le t\le\tfrac12,\\
 K_{2t-1}(H_1(x)),&\tfrac12\le t\le1
 \end{cases}
\]
is a strong deformation retraction of $P$ onto $\mathcal D$: its
second branch is defined by $H_1(P)\subset U$, both branches stay
in $P$, and every $G_t$ fixes $\mathcal D$.
In particular, $P$ is connected. Mayer--Vietoris for the tree of
spheres \cite[Section~2.2]{Hatcher} gives
\[
 H_2(P;\Z)=\bigoplus_{j=1}^r\Z[D_j],\qquad
 H_1(P;\Z)=H_3(P;\Z)=H_4(P;\Z)=0.
\]
The intersection form on $H_2(P;\R)$ has matrix $(D_i\cdot D_j)$
in the real images of this basis. The asserted Betti numbers,
Euler characteristic and signature follow.

\smallskip\noindent
\emph{The K\"ahler form and strong convexity.}
With the K\"ahler form $\eta$ fixed above,
\begin{equation}\label{eq:resolutionkahler}
 \Omega_\epsilon=dd^cf+\epsilon\eta,\qquad \epsilon>0,
\end{equation}
is K\"ahler, with each $D_j$ holomorphic and symplectic.

The radial derivative above shows that
$\{\sum_\nu|p_\nu|^2=c\}$ meets each ray in $\C^2\setminus\{0\}$
once and transversely. Radial parametrization and $\pi$ therefore
identify $\partial P$ with $S^3/\Gamma$.
For $0<\delta<c$, the gradient flow gives a two-sided collar
\[
 A=\{c-\delta<f<c+\delta\}
 \cong(c-\delta,c+\delta)\times\partial P.
\]
The finite covering $S^3\to S^3/\Gamma$ and transfer
\cite[Section~3.G]{Hatcher} give $H^2(A;\R)=0$.
Choose a one-form $\theta$ on $A$ with $d\theta=\eta|_A$, and set
\[
 \lambda_\epsilon=d^cf+\epsilon\theta,\qquad
 d\lambda_\epsilon=\Omega_\epsilon|_A.
\]
At $\epsilon=0$, let $g_0$ be the K\"ahler metric of $dd^cf$ on $A$.
The Liouville vector field determined by
$\iota_{Z_0}dd^cf=d^cf$ is
\[
 Z_0=\tfrac12\nabla_{g_0}f,\qquad
 df(Z_0)=\tfrac12|\nabla_{g_0}f|^2>0
 \quad\text{on }\partial P.
\]
The solutions of $\iota_{Z_\epsilon}\Omega_\epsilon=\lambda_\epsilon$
depend smoothly on $\epsilon$. Compactness of $\partial P$ therefore
gives $\epsilon_*>0$ such that $df(Z_\epsilon)>0$ for
$0\le\epsilon\le\epsilon_*$ along $\partial P$. With
$\beta_\epsilon=\lambda_\epsilon|_{\partial P}$,
\[
 \bigl(\iota_{Z_\epsilon}\Omega_\epsilon^2\bigr)|_{\partial P}
 =2\beta_\epsilon\wedge d\beta_\epsilon>0
\]
in the boundary orientation. Fix $0<\epsilon\le\epsilon_*$.
Then $(P,\Omega_\epsilon)$ is strongly convex.

\smallskip\noindent
\emph{The standard contact boundary.}
Gray stability \cite[Theorem~10.5]{CdS} applied to
$\{\beta_\epsilon\}_{0\le\epsilon\le\epsilon_*}$ identifies the
boundary contact structure with $\ker(d^cf|_{\partial P})$.
To identify this structure, consider the invariant functions
\[
 F_s(z)=(1-s)|z|^2+s\sum_\nu|p_\nu(z)|^2,\qquad 0\le s\le1.
\]
Each $F_s$ is strictly plurisubharmonic off the origin and increases
strictly from zero to infinity along every ray. The implicit function
theorem gives smooth $\Gamma$-equivariant radial parametrizations
$\Psi_s:S^3\to\{F_s=c\}$. The contact forms
\[
 \alpha_s=\Psi_s^*(d^cF_s|_{\{F_s=c\}})
\]
descend to the fixed quotient $S^3/\Gamma$; this is the standard
strictly pseudoconvex boundary construction
\cite[Section~3.A]{CNPP}. At $s=0$ their kernel is
$\xi_{\mathrm{std}}$, and at $s=1$ the form is the pullback of
$d^cf|_{\partial P}$. A second application of Gray stability
\cite[Theorem~10.5]{CdS} identifies these endpoints. Both isotopies
pull back the final contact form to a positive multiple of the initial
one, so their composition gives the required coorientation-preserving
contactomorphism.
\end{proof}

The next lemma supplies the discrepancy inequality used in the filling
obstruction. Its strict form is essential when $W$ is a rational
homology ball.

\begin{lemma}\label{lem:discrepancies}
For the domain in Lemma~\ref{lem:resolutionfill}, let
$S=\sum_{j=1}^r a_j$.  There exist numbers $0\le d_j<1$ such that
\begin{equation}\label{eq:kappap}
\kappa_P=-\sum_{j=1}^r d_j(a_j-2).
\end{equation}
If some $a_j>2$, then $\kappa_P>2r-S$.  If all $a_j=2$, then
$\kappa_P=0$ and $S=2r$.
\end{lemma}

\begin{proof}
Write $\pi:\widetilde V\to V=\C^2/\Gamma$ for the resolution,
$q:\C^2\to V$ for the quotient map, and $N=\partial P$; set $o=q(0)$.
Let $Q=(D_i\cdot D_j)$, $A=-Q$ and $v=(a_j-2)_{j=1}^r$.
By Lemma~\ref{lem:resolutionfill}, $A$ is positive definite with
nonpositive off-diagonal entries, and $v\ge0$.

\smallskip\noindent
\emph{Nonnegative coefficients.}
Define $d=A^{-1}v$. For $s>\lambda_{\max}(A)$, the matrix $B=sI-A$
is entrywise nonnegative and $\|B/s\|_2<1$. Hence
\[
 A^{-1}=\frac1s\sum_{k=0}^\infty(B/s)^k
\]
is entrywise nonnegative, so $d\ge0$.
The tangent--normal exact sequence and
$N_{D_j/\widetilde V}\cong\mathcal O_{\widetilde V}(D_j)|_{D_j}$
(see \cite[Tag~0B3P]{Stacks}) give the adjunction identities
\[
 \langle c_1(TP),[D_j]\rangle=2-a_j=-v_j,
 \qquad \deg(K_{\widetilde V}|_{D_j})=v_j.
\]

\smallskip\noindent
\emph{The strict bound.}
Set $m=|\Gamma|$. Since $(\det\gamma)^m=1$ for $\gamma\in\Gamma$,
the tensor $\sigma_0=(dz_1\wedge dz_2)^{\otimes m}$ descends to a
nowhere-vanishing pluricanonical section $\sigma$ on
$V^*=V\setminus\{o\}$. We first exhibit its pullback as a meromorphic
section on the entire resolution.

Among the invariant generators used in Lemma~\ref{lem:resolutionfill},
choose $p_1,p_2$ whose differentials are independent at some nonzero
point; such a pair exists because the generators embed the smooth
quotient. Their Jacobian
$\Delta=\det\bigl(\partial(p_1,p_2)/\partial(z_1,z_2)\bigr)$ is not
identically zero and satisfies
$\Delta(\gamma z)\det\gamma=\Delta(z)$.
Thus $\Delta^m$ descends to a regular function $G$ on $V$.
Writing $\bar p_i$ for the descended generators and
$F_i=\bar p_i\circ\pi$, define
\[
 \tau=\frac{(dF_1\wedge dF_2)^{\otimes m}}{G\circ\pi}.
\]
This is a global meromorphic section of $K_{\widetilde V}^{\otimes m}$.
The quotient tensor $(d\bar p_1\wedge d\bar p_2)^{\otimes m}/G$
pulls back by $q$ to $\sigma_0$ on $\{\Delta\ne0\}$.
The meromorphic identity theorem therefore yields
$\tau|_{\widetilde V\setminus\mathcal D}=\pi^*\sigma$.
Consequently $\operatorname{div}(\tau)=\sum_j n_jD_j$ for integers
$n_j$, and the divisor--line-bundle correspondence
\cite[Tag~02SE]{Stacks} gives
\[
 K_{\widetilde V}^{\otimes m}
 \cong\mathcal O_{\widetilde V}\Bigl(\sum_j n_jD_j\Bigr).
\]
Restricting this line-bundle isomorphism to each $D_i$ and taking
degrees yields $Qn=mv$. Since $Ad=v$ and $Q=-A$, we obtain
$n=-md$.

For a local pluricanonical tensor
$h(dw_1\wedge dw_2)^{\otimes m}$, use the intrinsic density
$\nu=|h|^{2/m}dV_w$, as in \cite[Definition~3.2]{FG}, normalized so
that $\nu_{\sigma_0}$ is Euclidean volume. For $R>0$, let $U=q(B_R(0))$.
The covering $B_R(0)\setminus\{0\}\to U\setminus\{o\}$ has degree $m$,
while $\pi$ is biholomorphic off $\mathcal D$. Change of variables
therefore gives
\[
 \int_{\pi^{-1}(U)\setminus\mathcal D}\nu_\tau
 =\int_{U\setminus\{o\}}\nu_\sigma
 =\frac{\Vol(B_R(0))}{m}<\infty.
\]
Near a general point of $D_j$, take coordinates $(u,w)$ with
$D_j=\{u=0\}$. Then
$\tau=u^{n_j}h(u,w)(du\wedge dw)^{\otimes m}$ with $h$ a holomorphic
unit. On a sufficiently small bidisc, $|h|$ is bounded above and
below by positive constants, so finite mass forces
\[
 \int_0^\varepsilon r^{2n_j/m+1}\,dr<\infty,
 \qquad\text{hence}\qquad \frac{n_j}{m}>-1.
\]
At equality the integral diverges logarithmically. Thus
$d_j=-n_j/m<1$, proving $0\le d_j<1$.

\smallskip\noindent
\emph{The relative Chern square.}
Since $N\cong S^3/\Gamma$, we have $H^1(N;\R)=H^2(N;\R)=0$.
The exact sequence of the pair and Poincar\'e--Lefschetz duality
\cite[Sections~3.1--3.3]{Hatcher} give an isomorphism
\[
 \iota:H^2(P,N;\R)\longrightarrow H^2(P;\R)
\]
and the basis $\Theta_j=\PD[D_j]\in H^2(P,N;\R)$.
Put $c=c_{1,\R}(TP)\in H^2(P;\R)$ and $\widehat c=\iota^{-1}c$.
Since $\langle\iota\Theta_j,[D_i]\rangle=Q_{ij}$ and $Qd=-v$,
the class $\iota(\sum_jd_j\Theta_j)$ and $c$ have the same
evaluations on the homology basis $[D_i]$.
Thus $\widehat c=\sum_jd_j\Theta_j$, and the mixed cup product gives
\[
 \kappa_P
 =\langle\widehat c\smile c,[P,N]\rangle
 =d^TQd=-d^Tv
 =-\sum_{j=1}^r d_j(a_j-2).
\]
Finally,
\[
 \kappa_P-(2r-S)=\sum_{j=1}^r(1-d_j)(a_j-2).
\]
Every summand is nonnegative, and one is positive if some $a_j>2$.
If all $a_j=2$, then $v=d=0$, so $\kappa_P=0$ and $S=2r$.
\end{proof}

\begin{proof}[Proof of Theorem~\ref{thm:obstruction}]
Suppose that such a domain $W$ exists. Put
$Y=\partial W$, $C=M\setminus\operatorname{int}W$ and $b=b_2(W)$.

\smallskip\noindent
\emph{The topology of the filling.}
Let $\iota_W:W\hookrightarrow M$ and $\iota_C:C\hookrightarrow M$
be the inclusions. Since $Y$ is a rational homology sphere,
$H_1(Y;\R)=H_2(Y;\R)=0$, and Mayer--Vietoris
\cite[Chapters~2--3]{Hatcher} gives
\[
 0=H_2(Y;\R)\longrightarrow H_2(W;\R)\oplus H_2(C;\R)
 \xrightarrow{(\iota_W)_*+(\iota_C)_*}H_2(M;\R)
 \longrightarrow H_1(Y;\R)=0.
\]
In particular, $(\iota_W)_*$ is injective.
The exact sequence of $(W,Y)$, Poincar\'e--Lefschetz duality
and the universal coefficient theorem give
\[
 H_2(W;\R)\xrightarrow{\;\cong\;}H_2(W,Y;\R)
 \xrightarrow{\;\cong\;}H^2(W;\R)
 \xrightarrow{\;\cong\;}H_2(W;\R)^*.
\]
This composite sends $x$ to $Q_W(x,\cdot)$, where $Q_W$ is the
intersection form of $W$; hence $Q_W$ is nondegenerate.
Representing cycles in the interiors of the two pieces gives,
for $x,y\in H_2(W;\R)$ and $z\in H_2(C;\R)$,
\[
 \begin{aligned}
 Q_M((\iota_W)_*x,(\iota_W)_*y)&=Q_W(x,y),\\
 Q_M((\iota_W)_*x,(\iota_C)_*z)&=0,
 \end{aligned}
\]
where $Q_M$ is the intersection form of $M$.
Thus $(\iota_W)_*H_2(W;\R)$ is a nondegenerate subspace of $H_2(M;\R)$
contained in $[F]^\perp$, since $F\subset C$.
The class $[F]$ is nonzero because $F$ has positive
symplectic area. The signature of $M$ and $F^2=0$ imply that
$[F]^\perp$ is negative semidefinite, with radical $\R[F]$ and
negative index $\delta$. Thus the intersection form of $W$ is
negative definite, with rank zero allowed, and
\begin{equation}\label{eq:fillingnegative}
 0\le b\le\delta\le1,\qquad \kappa_W\le0.
\end{equation}
In the exact sequence of $(W,Y)$, the boundary map
$H_4(W,Y;\R)\to H_3(Y;\R)$ is an isomorphism on fundamental classes,
whereas $H_3(W,Y;\R)\cong H^1(W;\R)=0$.
Hence $H_3(W;\R)=0$. Also $H_4(W;\R)=0$, since $W$ is connected with
nonempty boundary. Consequently
\begin{equation}\label{eq:fillingeulersignature}
 \chi(W)=1+b,\qquad \sigma(W)=-b.
\end{equation}

\smallskip\noindent
\emph{Symplectic replacement.}
Let $(P,\Omega_P)$ be the domain from
Lemma~\ref{lem:resolutionfill}. Choose a coorientation-preserving
contactomorphism $\varphi:\partial P\to Y$ and boundary contact
forms $\alpha_P,\alpha_W$ induced by Liouville primitives. Then
$\varphi^*\alpha_W=h\alpha_P$ for a smooth positive function $h$.
The collar homotopy formula extends $\alpha_W$ to a primitive of
$\Omega$ on both sides of $Y$. Its Liouville field has the same
outward normal component along $Y$. The Liouville flow then gives
the collar model $d(e^t\alpha)$, with $t$ increasing outward.

Rescale $\Omega_P$ by $0<a<e^{-1}\min h$.
Choose a smooth nondecreasing function $\chi:[0,1]\to[0,1]$
equal to zero near $0$ and to one near $1$, and set
\[
 k(s,y)=a\exp\!\left(s+\chi(s)\bigl(\log(h(y)/a)-1\bigr)\right).
\]
It equals $ae^s$ near $0$ and $h(y)e^{s-1}$ near $1$, and
$\partial_s\log k=1+\chi'(s)(\log(h/a)-1)>0$.
Therefore
\begin{equation}\label{eq:contactcylinder}
 \Omega_{\mathrm{col}}=d(k\alpha_P),\qquad
 \Omega_{\mathrm{col}}^2
 =2k\partial_s k\,ds\wedge\alpha_P\wedge d\alpha_P>0.
\end{equation}
This joins the full collar forms on $[0,1]\times\partial P$,
using $t=s$ at the $P$ end and $t=s-1$ at the $C$ end.
We obtain a closed symplectic manifold
$\widehat M=C\cup_Y P$, with the inserted collar understood.
The symplectic form agrees with $\Omega$ on $C$ and with
$a\Omega_P$ near $\mathcal D$. Thus the actual sphere $F$ is
preserved, and the resolution structure remains compatible
near $\mathcal D$.

\smallskip\noindent
\emph{Compatible almost complex structures.}
Denote the glued form by $\widehat\Omega$.
For a symplectic vector bundle $(E,\omega)$ with a smooth positive
definite metric $g$, set
\[
 g(Au,v)=\omega(u,v),\qquad
 B=(-A^2)^{1/2},\qquad J=B^{-1}A.
\]
Since $A$ is invertible and $g$-skew-adjoint, $-A^2=A^*A$ is
positive definite, and its positive square root $B$ is smooth and
commutes with $A$. Thus $J^2=-I$ and
\[
 \omega(Ju,Jv)=\omega(u,v),\qquad
 \omega(u,Jv)=g(u,Bv),
\]
so $J$ is compatible.

Choose neighbourhoods
\[
 F\subset V_F\Subset U_F,\qquad
 \mathcal D\subset V_D\Subset U_D,\qquad U_F\cap U_D=\varnothing.
\]
On $U_D$ take the resolution structure $J_D$.
Along $F$, choose a compatible direct-sum structure $j_F$ on
$TF\oplus(TF)^{\widehat\Omega}$.
Extend $\widehat\Omega(\,\cdot\,,j_F\,\cdot\,)$ to a positive
definite metric on $U_F$ and apply the formula above to obtain a
compatible $J_F$ with $J_F|_F=j_F$.
For $\alpha\in\{F,D\}$, set
\[
 g_\alpha=\widehat\Omega(\,\cdot\,,J_\alpha\,\cdot\,)
 \quad\text{on }U_\alpha.
\]
Choose an arbitrary global positive definite metric $g_*$ and
cutoffs
\[
 \rho_\alpha\in C_c^\infty(U_\alpha),\qquad
 0\le\rho_\alpha\le1,\qquad \rho_\alpha|_{V_\alpha}=1,
\]
and define
\[
 g=(1-\rho_F-\rho_D)g_*+\rho_Fg_F+\rho_Dg_D.
\]
The local terms are extended by zero. Since $\rho_F\rho_D=0$,
the coefficients are nonnegative and sum to one, so $g$ is a
smooth positive definite metric. Define
\[
 g(\widehat A u,v)=\widehat\Omega(u,v),\qquad
 \widehat J=(-\widehat A^2)^{-1/2}\widehat A.
\]
This structure is $\widehat\Omega$-compatible, and on each
$V_\alpha$,
\[
 g=g_\alpha,\qquad \widehat A=J_\alpha,\qquad
 (-\widehat A^2)^{1/2}=I,\qquad \widehat J=J_\alpha.
\]
Hence $F$ and all $D_j$ are $\widehat J$-holomorphic.

For any two compatible structures $J_0,J_1$, applying the same
formula to
\[
 g_t=(1-t)\widehat\Omega(\,\cdot\,,J_0\,\cdot\,)
       +t\widehat\Omega(\,\cdot\,,J_1\,\cdot\,),
 \qquad 0\le t\le1,
\]
gives a homotopy from $J_0$ to $J_1$, fixed wherever they agree.
Thus the compatible first Chern class is independent of the choice.

\smallskip\noindent
\emph{Chern numbers.}
The vanishing of $H^1(Y;\R)$ and $H^2(Y;\R)$ gives orthogonal decompositions
\[
 \begin{aligned}
 H^2(M;\R)&\cong H^2(C;\R)\oplus H^2(W;\R),\\
 H^2(\widehat M;\R)&\cong H^2(C;\R)\oplus H^2(P;\R).
 \end{aligned}
\]
Indeed, each class on a piece has a unique relative lift represented
by a closed form vanishing near $Y$; zero extensions from opposite
pieces have zero cup product. Thus
\begin{equation}\label{eq:replacementsignature}
 b_2^+(\widehat M)=1,\qquad
 b_2^-(\widehat M)=1+\delta+r-b,
\end{equation}
and $\chi(\widehat M)-\chi(M)=r-b$,
$\sigma(\widehat M)-\sigma(M)=-r+b$.
The closed almost complex identity $c_1^2=2\chi+3\sigma$
\cite[Lecture~12, Corollary~1]{Auroux} yields
\[
 c_1(\widehat M)^2-c_1(M)^2=-(r-b).
\]
Both symplectic forms restrict to $\Omega|_C$ on the common exterior.
The compatible structures induced there are homotopic by the metric
construction above, so their first Chern classes coincide.
The orthogonal decompositions therefore give
$c_1(M)^2=\kappa_C+\kappa_W$ and
$c_1(\widehat M)^2=\kappa_C+\kappa_P$.
Comparing these expressions proves
\begin{equation}\label{eq:chernreplacement}
 \kappa_W=\kappa_P+r-b.
\end{equation}

\smallskip\noindent
\emph{A genus-zero member containing the exceptional divisor.}
For the structure $\widehat J$ fixed above, write
$e=[F]\in H_2(\widehat M;\Z)$ and
$c=c_1(T\widehat M,\widehat J)\in H^2(\widehat M;\Z)$.
Positivity of intersections \cite[Theorem~2.123]{WendlLectures}
and $F^2=0$ show that $e$ is $\widehat J$-nef.
Adjunction gives $c\cdot e=2$, so
$g_{\widehat J}(e)=1+(e^2-c\cdot e)/2=0$.

Apply the embedded-sphere case of
\cite[Theorem~4.1.4]{Markov} to $F$: every point of $\widehat M$
lies on a $\widehat J$-holomorphic subvariety of total class $e$.
By \cite[Theorem~1.5]{LZnef}, these subvarieties are connected
trees of embedded rational curves; its hypotheses hold because
$e$ is nef of genus zero and $\widehat J$ is tamed by the symplectic
form.

Choose such a member through a point of $D_1$.
If $D_1$ were not a component, positivity of intersections would
give $e\cdot[D_1]>0$, contrary to $F\cap D_1=\varnothing$.
The same argument propagates along the connected divisor:
a component adjacent to one already in the member must also
belong to it. Hence the member has the form
\[
 \Theta=\sum_{i=1}^n m_iA_i,\qquad m_i\in\Z_{>0},\qquad
 A_j=D_j\ (1\le j\le r),\qquad n=r+\ell.
\]
Negative definiteness of the exceptional matrix and
$[\Theta]^2=e^2=0$ imply $\ell\ge1$.
Moreover, $\Theta$ cannot contain $F$, since its total symplectic
area already equals the area of $F$ and the other components have
positive area. Therefore all $F\cdot A_i\ge0$, and
$0=e^2=\sum_i m_i(F\cdot A_i)$ forces $F\cap A_i=\varnothing$
for every $i$. In particular,
\[
 0=e\cdot[A_i]=m_iA_i^2+\sum_{j\ne i}m_j(A_j\cdot A_i).
\]
Connectedness and $n\ge2$ make the last sum positive, so every
$A_i^2<0$.

\smallskip\noindent
\emph{The number of additional components.}
Let $Q_\Theta=(A_i\cdot A_j)$. The preceding identities give
$Q_\Theta m=0$, and for $x\in\R^n$,
\begin{equation}\label{eq:weightedlaplacian}
 x^TQ_\Theta x
 =-\sum_{i<j}(A_i\cdot A_j)m_im_j
 \left(\frac{x_i}{m_i}-\frac{x_j}{m_j}\right)^2.
\end{equation}
Since the graph is connected, this form has kernel $\R m$ and
negative index $n-1$. All component classes lie in $e^\perp$.
By~\eqref{eq:replacementsignature},
$e^\perp/\R e$ is negative definite of dimension $\delta+r-b$.
The map $x\mapsto\sum_i x_i[A_i]$, followed by projection to this
quotient, is injective on any negative definite subspace for
$Q_\Theta$: a vector mapping into $\R e$ has square zero.
Comparing dimensions yields
\begin{equation}\label{eq:componentbudget}
 n-1\le\delta+r-b,\qquad 1\le\ell\le\delta+1-b.
\end{equation}

\smallskip\noindent
\emph{The total weight of the member.}
We use the configuration theorem
\cite[Definitions~3.4--3.6 and~4.2, Proposition~4.6]{LZnef}.
In the form needed here, a connected configuration with at least
two vertices, labelled by effective integral classes $u_i$ with
positive integer multiplicities,
$\operatorname{adj}(u_i):=u_i^2-c\cdot u_i\ge-2$, and positive
nonzero edge pairings is a tree if its total class is nef of
genus zero. Every edge then has intersection number one.
The vertices may be effective classes represented by reducible
subvarieties.

Start with $u_i=[A_i]$, retaining the multiplicities $m_i$.
Then $\operatorname{adj}(u_i)=-2$,
$\sum_i m_i u_i=e$ and $e\cdot u_i=0$.
We reduce this configuration while preserving these properties.
Whenever at least two vertices remain, connectedness and
$e\cdot u_i=0$ imply $u_i^2<0$, whereas
\[
 2=c\cdot e=\sum_i m_i(2+u_i^2)
\]
forces a vertex $v$ with $v^2=-1$.
If its neighbours are $u_1,\ldots,u_k$, then
$e\cdot v=0$ gives $m_v=\sum_{i=1}^k m_i$.
Delete $v$ and replace each neighbour by $u_i+v$, keeping its
multiplicity. The new classes are effective, their weighted sum
is still $e$, and they remain orthogonal to $e$. The tree
pairings give
\[
 \operatorname{adj}(u_i+v)=-2,\qquad
 (u_i+v)^2=u_i^2+1,\qquad
 (u_i+v)\cdot(u_j+v)=1\quad(i\ne j).
\]
All other pairings are unchanged. The new configuration is
connected and satisfies the same effectivity and numerical
conditions. If $k\ge3$, it has at least three vertices, and the
modified neighbours form a triangle, contrary to the configuration
theorem. Hence $k=1$ or $2$, and the reduction can continue.

Each step deletes one vertex and decreases the unweighted sum
of negative self-intersections by $1+k\le3$.
After $n-1$ steps a single class $u$ remains with $e=mu$,
so $u^2=0$. The last step has $k=1$ and decreases the sum by $2$.
It follows that
\begin{equation}\label{eq:weightboundall}
 S_\Theta:=\sum_{i=1}^n(-A_i^2)\le3(n-2)+2=3n-4.
\end{equation}
Each of the $\ell$ additional components has negative integral
square and contributes at least one. Thus, with
$S=\sum_{j=1}^r a_j$, equations~\eqref{eq:componentbudget}
and~\eqref{eq:weightboundall} imply
\begin{equation}\label{eq:resolutionweightbound}
 S\le S_\Theta-\ell\le3r+2\ell-4
 \le3r+2\delta-2b-2.
\end{equation}

\smallskip\noindent
\emph{The numerical contradiction.}
If some $a_j>2$, Lemma~\ref{lem:discrepancies},
\eqref{eq:chernreplacement} and~\eqref{eq:resolutionweightbound}
give
\[
 \kappa_W=\kappa_P+r-b
 >3r-S-b
 \ge b+2-2\delta
 \ge0,
\]
contradicting~\eqref{eq:fillingnegative}.
The inequality is strict also when $b=0$ and $\delta=1$.
If all $a_j=2$, then $\kappa_P=0$ and $S=2r$ by
Lemma~\ref{lem:discrepancies}. Thus
$0\ge\kappa_W=r-b$ and $1\le r\le b\le\delta\le1$,
forcing $r=b=\delta=1$.
But~\eqref{eq:resolutionweightbound} gives $S\le1$, whereas
$S=2$. This final contradiction proves the theorem.
\end{proof}

\begin{proposition}\label{prop:filling}
Let $(X,\omega)$ be a connected compact K\"ahler surface and
$f:X\to\Sigma$ a holomorphic map with connected fibres onto a smooth
compact curve.  Suppose that the generic fibre is $\PP^1$ and every
singular fibre is a reduced union of two $(-1)$-curves meeting
transversely at one point.  Let $D\Subset\Sigma$ be an open disc whose
closure contains $\delta\in\{0,1\}$ singular values and whose boundary
contains none, and put $U=f^{-1}(D)$.
There is no connected compact symplectic domain $W\Subset U$ with
connected strongly convex boundary contactomorphic to
$(S^3/\Gamma,\xi_{\mathrm{std}})$ for a nontrivial finite subgroup
$\Gamma\subset U(2)$ acting freely on $S^3$.
\end{proposition}

\begin{proof}
Suppose that $W$ exists.  By Lemma~\ref{lem:localtop}, $H_1(U;\R)=0$.
The Mayer--Vietoris sequence for
$U=W\cup(U\setminus\operatorname{int}W)$, using
$H_1(\partial W;\R)=H_2(\partial W;\R)=0$, gives $H_1(W;\R)=0$.

In each singular fibre outside $\overline U$, blow down one of its
$(-1)$-components symplectically.  The chosen spheres are mutually
disjoint, and the operations can be supported away from $\overline U$.
Also preserve a smooth symplectic fibre $F$ over a point outside
$\overline D$.  The resulting closed symplectic manifold $M$ is an
$S^2$-bundle over $\Sigma$ blown up $\delta$ times.  Therefore
\[
b_2^+(M)=1,\qquad b_2^-(M)=1+\delta.
\]
The original $W$ and its boundary structure are unchanged, and
$F\subset M\setminus W$ is an actual embedded symplectic sphere with
$F^2=0$.  Theorem~\ref{thm:obstruction} applies and gives a contradiction.
\end{proof}

\section[Application to K\"ahler--Ricci Flow on compact K\"ahler surfaces]
{Application to K\"ahler--Ricci Flow on compact K\"ahler \\ surfaces}\label{sec:limits}

Throughout this section $\alpha\ne0$, and $f:X\to\Sigma$ is the
fibration in Proposition~\ref{prop:reduction}.

\subsection{The compact side of a contact hypersurface}

\begin{lemma}\label{lem:compactside}
Let $U$ be as in \eqref{eq:localU}. A smooth closed connected orientable
embedded hypersurface $Y\subset U$ separates $U$, and exactly one
component of $U\setminus Y$ has compact closure. Denote that closure by
$W$. If $Y$ is a rational homology sphere, then $H_1(W;\R)=0$ and the
intersection form of $W$ is negative definite, with the zero-dimensional
case allowed.
\end{lemma}

\begin{proof}
\emph{The two sides and the compact closure.}
Choose an auxiliary metric. The orientations of $U$ and $Y$ give a
global unit normal $\nu$ along $Y$. Extend it to a vector field near
$Y$ and let $\varphi_t$ be its flow
\cite[Lemma~8.6 and Theorems~4.5, 9.12]{Lee}. For sufficiently small
$\epsilon>0$ this gives a two-sided collar
\[
 \Phi:Y\times(-2\epsilon,2\epsilon)\xrightarrow{\cong}N\subset U,
 \qquad \Phi(y,t)=\varphi_t(y).
\]
Indeed, $d\Phi_{(y,0)}(v,a)=v+a\nu_y$ is invertible, so compactness
and the inverse function theorem give a local diffeomorphism on a
uniform collar. Failure of injectivity on arbitrarily small collars
would give colliding pairs converging to the same point of
$Y\times\{0\}$, contrary to local injectivity.

Suppose $U\setminus Y$ were connected. Fix $y_0\in Y$ and choose
a piecewise smooth path $c$ in the complement from $\Phi(y_0,\epsilon)$ to
$\Phi(y_0,-\epsilon)$. Its compact image misses a thinner collar
$\Phi(Y\times[-\eta,\eta])$, where $0<\eta<\epsilon$.
Choose $\rho\in C_c^\infty((-\eta,\eta))$ with $\int\rho\,dt=1$,
and define $\theta$ by $\Phi^*\theta=\rho(t)\,dt$ on $N$ and
zero elsewhere. For the loop $\gamma$ consisting of the transverse
collar arc followed by $c$,
\[
 d\theta=0,\qquad
 \int_\gamma\theta=\int_{-\epsilon}^{\epsilon}\rho(t)\,dt=1.
\]
This contradicts simple connectivity of $U$
(Lemma~\ref{lem:localtop}), since a closed one-form is then exact
\cite[Corollary~16.27]{Lee}.

A path from any complementary component to $Y$, stopped at its first
intersection with $Y$, shows that the component meets one of
$\Phi(Y\times(0,2\epsilon))$ and
$\Phi(Y\times(-2\epsilon,0))$.
Each collar side is connected, so there are exactly two components,
both with boundary $Y$.
The model in Lemma~\ref{lem:localtop} supplies a compact $K\subset U$
with $Y\subset\operatorname{int}K$ and
\[
 U\setminus K\cong S^2\times
       \bigl(\R^2\setminus\overline{B_R(0)}\bigr).
\]
This connected, non-relatively-compact set lies in one component;
the other has closure contained in $K$. Thus precisely the latter
closure $W$ is compact. Put $V=U\setminus\operatorname{int}W$.
Reversing the collar coordinate if necessary,
\[
 \Phi^{-1}(W\cap N)=Y\times(-2\epsilon,0],\qquad
 \Phi^{-1}(V\cap N)=Y\times[0,2\epsilon).
\]
Thus $W$ is a compact smooth manifold with boundary $Y$, oriented by $U$.

\smallskip\noindent
\emph{Homology and the intersection form.}
Assume now that $Y$ is a rational homology sphere. Set
\[
 N_0=\Phi(Y\times(-\epsilon,\epsilon)),\qquad
 \mathcal A=\operatorname{int}W\cup N_0,\qquad
 \mathcal B=\operatorname{int}V\cup N_0.
\]
These form an open cover of $U$, with
$\mathcal A\cap\mathcal B=N_0\simeq Y$.
The collar contractions $t\mapsto t-s\max(t,0)$ and
$t\mapsto t-s\min(t,0)$, extended by the identity elsewhere,
give $\mathcal A\simeq W$ and $\mathcal B\simeq V$.
Mayer--Vietoris \cite[Section~2.2]{Hatcher} therefore gives
\[
 \begin{aligned}
 0&\longrightarrow H_2(W;\R)\oplus H_2(V;\R)
   \longrightarrow H_2(U;\R)\longrightarrow0,\\
 0&\longrightarrow H_1(W;\R)\oplus H_1(V;\R)
   \longrightarrow H_1(U;\R)=0.
 \end{aligned}
\]
Hence $H_1(W;\R)=0$, and inclusion $i:W\hookrightarrow U$ induces
an injection $H_2(W;\R)\hookrightarrow H_2(U;\R)$.

Since $H^1(Y;\R)=H^2(Y;\R)=0$, the pair sequence, Poincar\'e--Lefschetz
duality $D_W$ and evaluation give isomorphisms
\cite[Section~3.1 and Theorem~3.43]{Hatcher}
\[
 H_2(W;\R)\xrightarrow{D_W^{-1}}H^2(W,Y;\R)
 \xrightarrow{\cong}H^2(W;\R)
 \xrightarrow{\mathrm{ev}}H_2(W;\R)^*.
\]
Their composite is $x\mapsto Q_W(x,\cdot)$, so $Q_W$ is
nondegenerate.

For $x,y\in H_2(W;\R)$, choose compactly supported closed two-forms $\eta_x,\eta_y$ on
$\operatorname{int}W$ representing $D_W^{-1}x,D_W^{-1}y$,
under $H_c^2(\operatorname{int}W;\R)\cong H^2(W,Y;\R)$,
and extend them by zero to $\widetilde\eta_x,\widetilde\eta_y$
on $U$. These extensions represent classes in $H_c^2(U;\R)$.
Compactly supported duality \cite[Theorem~3.35]{Hatcher}
and its de Rham realization \cite[Sections~5--6]{BottTu} give,
for every closed two-form $\beta$ on $U$,
\[
 \int_U\widetilde\eta_x\wedge\beta
 =\int_W\eta_x\wedge i^*\beta
 =\langle i^*[\beta],x\rangle
 =\langle[\beta],i_*x\rangle.
\]
Thus the extensions represent the compactly supported duals of
the included classes, and
\[
 Q_U(i_*x,i_*y)
 =\int_U\widetilde\eta_x\wedge\widetilde\eta_y
 =\int_W\eta_x\wedge\eta_y
 =Q_W(x,y).
\]
Lemma~\ref{lem:localtop} gives $Q_W(x,x)\le0$.
If $Q_W(x,x)=0$, Cauchy--Schwarz for $-Q_W$ yields
\[
 |Q_W(x,y)|^2\le Q_W(x,x)Q_W(y,y)=0
 \qquad\text{for all }y\in H_2(W;\R).
\]
Nondegeneracy forces $x=0$, proving negative definiteness.
\end{proof}

\begin{lemma}\label{lem:volume}
Let $(W,\Omega)$ be a compact symplectic four-manifold whose boundary
$Y$ is a rational homology sphere. Suppose $\Omega=d\lambda$ on a
collar of $Y$, and let $[\Omega]_{\mathrm{rel}}$ be the unique lift of
$[\Omega]\in H^2(W;\R)$ to $H^2(W,Y;\R)$.
With the boundary orientation on $Y$,
\begin{equation}\label{eq:volumeidentity}
 \int_W\Omega^2
 =\int_Y\lambda\wedge d\lambda+[\Omega]_{\mathrm{rel}}^2.
\end{equation}
If the intersection form of $W$ is negative definite and
$\lambda|_Y$ is a contact form, its Liouville vector field points
out of $W$.
\end{lemma}

\begin{proof}
The relative lift is unique because $H^1(Y;\R)=H^2(Y;\R)=0$.
Choose a cutoff $\chi$ supported in the collar and equal to one near
the boundary, and extend $\eta=\chi\lambda$ by zero. The closed form
$\Omega^0=\Omega-d\eta$ vanishes near $Y$ and represents this lift.
Stokes' theorem gives
\begin{align*}
 \int_W(\Omega^0)^2
 &=\int_W\Omega^2-2\int_Y\lambda\wedge\Omega
                  +\int_Y\lambda\wedge d\lambda\\
 &=\int_W\Omega^2-\int_Y\lambda\wedge d\lambda.
\end{align*}
This proves \eqref{eq:volumeidentity}. For the vector field $Z$ defined
by $\iota_Z\Omega=\lambda$, restriction to $Y$ gives
\[
 (\iota_Z\Omega^2)|_Y=2\lambda|_Y\wedge d(\lambda|_Y).
\]
Thus contactness makes $Z$ transverse to $Y$. Since $Y$ is connected,
$Z$ points either everywhere outwards or everywhere inwards. In the
latter case the boundary integral in \eqref{eq:volumeidentity} is
strictly negative. The relative square is nonpositive, because the
cohomological intersection form is dual to the negative definite
homological form. This contradicts $\int_W\Omega^2>0$, so $Z$ points
outwards.
\end{proof}

\subsection{Controlled primitives on an annulus}

\begin{lemma}\label{lem:primitive}
Let $Y$ be a closed oriented smooth manifold with $H^2(Y;\R)=0$, and let
$A=(a,b)\times Y$. Suppose closed two-forms $\Theta_i$ converge to zero
smoothly on $A$. On every product collar with compact closure in $A$
there are one-forms $\eta_i\to0$ smoothly with $d\eta_i=\Theta_i$.
\end{lemma}

\begin{proof}
Fix a compact interval in $(a,b)$ containing the collar in question,
choose $r_0$ in that interval, and fix a background metric on $Y$.
Write $\Theta_i=dr\wedge a_i(r)+b_i(r)$, with $a_i(r)$ and $b_i(r)$
forms on $Y$. Closedness says
$d_Yb_i=0$ and $\partial_rb_i=d_Ya_i$. Let $G$ be the Green operator
of the Hodge Laplacian on two-forms on $Y$. Since $b_i(r_0)$ is closed
and $H^2(Y;\R)=0$, the Hodge decomposition gives
\[
 c_i=d_Y^*G b_i(r_0),\qquad d_Yc_i=b_i(r_0).
\]
Set
\[
 \eta_i(r)=c_i+\int_{r_0}^r a_i(u)\,du.
\]
Direct differentiation gives $d\eta_i=\Theta_i$. Continuity of the
Hodge solution operator in the smooth topology, and integration over
the fixed compact interval, imply $\eta_i\to0$ smoothly.
\end{proof}

\begin{lemma}\label{lem:noannulus}
There do not exist times $s_i\to T$, scales $\rho_i\to0$, a nontrivial
finite subgroup $\Gamma\subset O(4)$ acting freely on $S^3$, and smooth
embeddings
\[
 \psi_i:\{1/2<|z|<2\}/\Gamma\longrightarrow X
\]
such that $\psi_i^*(\rho_i^{-2}g(s_i))$ converges smoothly locally to
the flat quotient metric.
\end{lemma}

\begin{proof}
\emph{The unitary limit.}
Suppose such embeddings exist and put
$A=\{1/2<|z|<2\}/\Gamma$. Set
\[
 h_i=\psi_i^*(\rho_i^{-2}g(s_i)),\qquad J_i=\psi_i^*J,
 \qquad \Omega_i=\psi_i^*(\rho_i^{-2}\omega(s_i)).
\]
Use the same notation for their lifts to the Euclidean covering
annulus. The K\"ahler condition gives
\[
 J_i^2=-\mathrm{Id},\qquad h_i(J_i\cdot,J_i\cdot)=h_i,
 \qquad \nabla^{h_i}J_i=0.
\]
In Cartesian coordinates, let $C_{i,a}$ be the connection matrices
of $h_i$. Since $h_i\to g_E$ smoothly locally,
\[
 C_{i,a}\longrightarrow0\text{ in }C^\infty_{\mathrm{loc}},
 \qquad
 \partial_aJ_i=J_iC_{i,a}-C_{i,a}J_i.
\]
Orthogonality bounds $J_i$ on compact subannuli; differentiating
this equation inductively gives
\[
 \|D^kJ_i\|_{C^0(K')}\longrightarrow0
 \qquad(k\ge1,\ K'\Subset\{1/2<|z|<2\}).
\]
Choose a convergent subsequence at one point. Integration along
paths in compact subannuli then gives $J_i\to J_0$ smoothly locally,
where $J_0$ is constant. Passing to the limit in the algebraic
identities and in deck equivariance yields
\[
 J_0^2=-\mathrm{Id},\qquad J_0^TJ_0=\mathrm{Id},\qquad
 \gamma J_0=J_0\gamma\quad(\gamma\in\Gamma).
\]
Choose one $Q\in O(4)$ with $Q^{-1}J_0Q=J_{\mathrm{std}}$.
Conjugating $\Gamma$ and precomposing every $\psi_i$ by this same
orthogonal map, we may assume
\[
 \Gamma\subset U(2),\qquad
 \Omega_i=h_i(J_i\cdot,\cdot)\longrightarrow
 \Omega_E=\sum_{j=1}^2dx_j\wedge dy_j
 \quad\text{smoothly locally}.
\]

\smallskip\noindent
\emph{Localization in one fibre neighbourhood.}
Fix $K=\{3/4\le |z|\le3/2\}/\Gamma$.
Radial and spherical paths, together with smooth metric convergence
on $K$, give a uniform bound for its intrinsic path diameter.
The lower bound \eqref{eq:schwarz} implies
\[
 \begin{gathered}
 |df(v)|_{g_\beta}\le c^{-1/2}|v|_{g(s_i)},\\
 \diam_{g_\beta}f(\psi_i(K))
 \le c^{-1/2}\rho_i\diam^{\mathrm{int}}_{h_i}(K)
 \le C\rho_i\longrightarrow0.
 \end{gathered}
\]
Fix $p_0\in K$ and pass to a subsequence with
$f(\psi_i(p_0))\to q\in\Sigma$. Choose a fixed disc $D$ containing
$q$ as in \eqref{eq:localU}. The diameter estimate gives
$\psi_i(K)\subset U=f^{-1}(D)$ for all large $i$.

\smallskip\noindent
\emph{The contact boundary and the filling contradiction.}
Put $Y=S^3/\Gamma$ and $Y_i=\psi_i(Y)$. Since
$\det_{\R}\gamma=|\det_{\C}\gamma|^2=1$, the standard orientation
of $S^3$ descends to $Y$. Finite-cover transfer
\cite[Section~3.G]{Hatcher} gives
\[
 H_k(Y;\mathbb Q)=0\quad(k=1,2),\qquad
 H^1(Y;\R)=H^2(Y;\R)=0.
\]
Thus $Y$ is a rational homology sphere, and
Lemma~\ref{lem:compactside} provides a compact domain
$W_i\Subset U$ with $\partial W_i=Y_i$ and negative definite
intersection form.

Identify $A=(1/2,2)\times Y$ radially and take
$C=(7/8,9/8)\times Y\subset\operatorname{int}K$.
Lemma~\ref{lem:primitive}, applied to the closed forms
$\Omega_i-\Omega_E$, gives
\[
 d\eta_i=\Omega_i-\Omega_E,\qquad
 \eta_i\longrightarrow0\quad\text{smoothly on }C.
\]
The $\Gamma$-invariant radial primitive
\[
 \lambda_E=\frac12\sum_{j=1}^2(x_j\,dy_j-y_j\,dx_j),
 \qquad d\lambda_E=\Omega_E
\]
descends to $C$. Set $\lambda_i=\lambda_E+\eta_i$ and
$\alpha_E=\lambda_E|_Y$, so that $d\lambda_i=\Omega_i$ and
$\ker\alpha_E=\xi_{\mathrm{std}}$.
For $\alpha_{i,t}=\alpha_E+t\eta_i|_Y$, $C^1$ convergence gives
\[
 \alpha_{i,t}\wedge d\alpha_{i,t}
 =(1+o(1))\alpha_E\wedge d\alpha_E>0
 \quad\text{uniformly on }[0,1]\times Y.
\]
Gray stability \cite[Theorem~10.5]{CdS} therefore supplies a
diffeomorphism $\phi_i:Y\to Y$ and a positive function $u_i$ with
\[
 \phi_i^*(\lambda_i|_Y)=u_i\alpha_E,\qquad u_i>0.
\]

Define $\widetilde\Omega_i$ globally on $X$ and
$\widetilde\lambda_i$ on $\psi_i(C)$ by
\[
 \widetilde\lambda_i=(\psi_i^{-1})^*\lambda_i,\qquad
 \widetilde\Omega_i=\rho_i^{-2}\omega(s_i),
 \qquad \beta_i=\widetilde\lambda_i|_{Y_i}.
\]
Then
\[
 d\widetilde\lambda_i=\widetilde\Omega_i
 \quad\text{on }\psi_i(C),\qquad
 (\psi_i|_Y\circ\phi_i)^*\beta_i=u_i\alpha_E.
\]
With the boundary orientation on $Y_i$, Lemma~\ref{lem:volume}
gives
\[
 \int_{Y_i}\beta_i\wedge d\beta_i
 =\int_{W_i}\widetilde\Omega_i^2
       -[\widetilde\Omega_i]_{\mathrm{rel}}^2>0,
\]
and shows that the Liouville field
$\iota_{Z_i}\widetilde\Omega_i=\widetilde\lambda_i$ points out of $W_i$.
Thus $Y_i$ is strongly convex.
The positive multiplier above identifies the induced contact
structure with $\xi_{\mathrm{std}}$, preserving coorientation.
Since $\widetilde\Omega_i$ is a global K\"ahler form on $X$,
$(W_i,\widetilde\Omega_i|_{W_i})$ contradicts
Proposition~\ref{prop:filling}.
\end{proof}

\subsection{Integral periods below the remaining-time scale}

\begin{proposition}[Exactness of smaller-scale limits]\label{prop:exact}
Let $s_i\to T$ and $r_i\to0$ satisfy
\begin{equation}\label{eq:scalegap}
 \frac{T-s_i}{r_i^2}\longrightarrow\infty.
\end{equation}
Suppose a connected K\"ahler manifold $(N,g_\infty,J_\infty)$ is a
smooth local limit of $r_i^{-2}g(s_i)$, with simultaneous convergence
of the pulled-back complex structures on an exhaustion of $N$.
Then its K\"ahler form $\Omega_\infty$ is exact:
\begin{equation}\label{eq:exact}
 [\Omega_\infty]=0\quad\text{in }H^2_{\mathrm{dR}}(N;\R).
\end{equation}
In particular $N$ contains no closed symplectic surface and no compact
complex curve.
\end{proposition}

\begin{proof}
\emph{Representatives and the Lipschitz bound.}
For the convergence embeddings $\psi_i$, simultaneous convergence
of the metrics and complex structures gives, smoothly locally,
\[
 \begin{aligned}
 h_i&=\psi_i^*(r_i^{-2}g(s_i))\longrightarrow g_\infty,\\
 \Omega_i&=\psi_i^*(r_i^{-2}\omega(s_i))
   =h_i(\psi_i^*J\,\cdot,\cdot)\longrightarrow\Omega_\infty.
 \end{aligned}
\]
Fix the closed form
\[
 \theta=\frac{1}{2\pi}\Ric(\omega_0),\qquad [\theta]=c_1(X),
\]
where $c_1(X)\in H^2(X;\Z)$ is identified with its image in
$H^2(X;\R)$, using the integral Chern-class normalization
\cite[Chapter~V, \S9; Chapter~VIII, \S6]{Demailly}.
By \eqref{eq:class} and Proposition~\ref{prop:reduction},
\[
 [\omega(t)]=f^*[\beta]+2\pi(T-t)c_1(X).
\]
Thus, for each $t$, there is a smooth real one-form $\zeta_t$ on $X$
such that
\[
 \omega(t)=f^*\beta+2\pi(T-t)\theta+d\zeta_t.
\]
Holomorphicity of $f$ and \eqref{eq:schwarz} give, for every $v\in TX$,
\[
 c|df(v)|_{g_\beta}^2
 =c(f^*\beta)(v,Jv)
 \le\omega(t)(v,Jv)=|v|_{g(t)}^2.
\]
This holds also at critical points of $f$. Integrating along any
piecewise smooth path $\gamma$ gives
\[
 L_{g_\beta}(f\circ\gamma)\le c^{-1/2}L_{g(t)}(\gamma).
\]

\smallskip\noindent
\emph{A compact set of connecting paths.}
Fix $p\in N$ and pass once to a subsequence such that
$q_i=f(\psi_i(p))\to q\in\Sigma$. Choose a fixed coordinate disc
$D\ni q$ and put $U=f^{-1}(D)$.

Let $A=\sum_{\ell=1}^m n_\ell\sigma_\ell$ be a nonzero smooth
integral singular two-cycle, where $n_\ell\in\Z$ and
$\sigma_\ell:\Delta^2\to N$.
Fix a vertex $v_0\in\Delta^2$ and choose piecewise smooth paths
$\gamma_\ell$ from $p$ to $\sigma_\ell(v_0)$.
Using the Euclidean metric on $\Delta^2$ and $g_\infty$ on $N$, set
\[
 \begin{aligned}
 K&=\bigcup_{\ell=1}^m
 \bigl(\sigma_\ell(\Delta^2)\cup\gamma_\ell([0,1])\bigr),\\
 L_A&=\max_{1\le\ell\le m}
 \left\{L_{g_\infty}(\gamma_\ell)
       +\diam(\Delta^2)\|d\sigma_\ell\|_{C^0(\Delta^2)}\right\}<\infty.
 \end{aligned}
\]
The set $K$ is compact and contains $p$ and $\supp A$.
For $z\in\Delta^2$, the path
\[
 t\longmapsto\sigma_\ell((1-t)v_0+tz)
\]
has length at most
$\diam(\Delta^2)\|d\sigma_\ell\|_{C^0(\Delta^2)}$.
Concatenation with $\gamma_\ell$, and initial segments of
$\gamma_\ell$, therefore join $p$ to every point of $K$ within
$K$ with length at most $L_A$.

For all large $i$, the embeddings are defined on $K$ and
$h_i\le2g_\infty$ there. For these connecting paths,
\[
 L_{g(s_i)}(\psi_i\circ\gamma)
 =r_iL_{h_i}(\gamma)\le\sqrt2\,r_iL_A.
\]
Consequently,
\[
 \sup_{x\in K}d_{g_\beta}(f(\psi_i(x)),q)
 \le d_{g_\beta}(q_i,q)+\sqrt{2/c}\,L_A r_i
 \longrightarrow0.
\]
Thus $\psi_i(K)\subset U$ for all sufficiently large $i$.

\smallskip\noindent
\emph{Vanishing of integral periods.}
Since $D$ is contractible, write $\beta|_D=d\xi_D$
\cite[Theorem~18.14]{Lee}. For the integral cycle
$B_i=\psi_{i\#}A\subset U$, Stokes' theorem gives
\[
 \int_{B_i}f^*\beta
 =\int_{B_i}d(f^*\xi_D)=0,\qquad
 \int_{B_i}d\zeta_{s_i}=0.
\]
The remaining term pairs $c_1(X)\in H^2(X;\Z)$ with
$[B_i]\in H_2(X;\Z)$:
\[
 k_i=\int_{B_i}\theta
 =\langle c_1(X),[B_i]\rangle\in\Z.
\]
The representative formula above therefore yields
\begin{equation}\label{eq:period}
 \int_A\Omega_i
 =r_i^{-2}\int_{B_i}\omega(s_i)
 =2\pi\frac{T-s_i}{r_i^2}\,k_i.
\end{equation}
Smooth convergence on the fixed finite cycle gives
$\int_A\Omega_i\to\int_A\Omega_\infty$. Hence \eqref{eq:scalegap}
implies
\[
 |k_i|=\frac{r_i^2}{2\pi(T-s_i)}
              \left|\int_A\Omega_i\right|\longrightarrow0,
 \qquad
 k_i=0\quad\text{for all }i\ge i_0(A).
\]
It follows that $\int_A\Omega_\infty=0$ for every smooth integral
two-cycle $A$. The threshold may depend on $A$.

Smooth singular homology and the de Rham theorem
\cite[Theorems~18.7 and~18.14]{Lee}, together with the universal
coefficient theorem \cite[Section~3.1]{Hatcher}, identify
\[
 H^2_{\mathrm{dR}}(N;\R)
 \cong\operatorname{Hom}(H_2(N;\Z),\R),
 \qquad
 [\Omega_\infty]\longmapsto
 \left([A]\longmapsto\int_A\Omega_\infty\right)=0.
\]
This proves \eqref{eq:exact} and gives a global smooth one-form
$\lambda$ with $\Omega_\infty=d\lambda$, also when $N$ is
noncompact or has infinitely generated homology.

Stokes' theorem now contradicts the positive area of a closed
symplectic surface. For a compact complex curve, take the
normalization $\nu:\widetilde C\to N$ of a reduced irreducible
component \cite[Chapter~II, Corollary~7.7 and Theorem~7.12]{Demailly}.
The form $\nu^*\Omega_\infty$ is nonnegative and is positive over
the regular locus, so
\[
 0<\int_{\widetilde C}\nu^*\Omega_\infty
   =\int_{\widetilde C}d(\nu^*\lambda)=0,
\]
again a contradiction.
\end{proof}

\section{Blow up analysis and the Type I estimate}\label{sec:tangents}

A \emph{shrinker} is a complete gradient shrinking Ricci soliton.
We extract a shrinking limit at a scale smaller than the remaining-time
scale.  In this section all Ricci-flow entropy and compactness results are
used in the Riemannian convention
$\bar g(u)=g(2u)$, so that $\partial_u\bar g=-2\Ric(\bar g)$.

\begin{proposition}\label{prop:twoscale}
Let $(X,J,g(t))$, $0\leq t<T<\infty$, be a K\"ahler--Ricci flow on a compact,
connected complex surface.  Suppose that
\[
 \sup_{X\times[0,T)}(T-t)|\Rm(g(t))|=\infty.
\]
Then there exist $s_j\to T$, $r_j\to0$, and a connected, complete gradient
K\"ahler shrinking Ricci soliton orbifold $(N,g_N,J_N,h)$, with at most isolated
quotient singularities, such that
\begin{equation}\label{eq:twoscale-ratio}
 \frac{T-s_j}{r_j^2}\longrightarrow\infty.
\end{equation}
There are an exhaustion $V_j$ of $N_{\reg}$ and smooth embeddings
$\psi_j:V_j\to X$ for which
\begin{equation}\label{eq:twoscale-convergence}
 \psi_j^*(r_j^{-2}g(s_j))\longrightarrow g_N,
 \qquad \psi_j^*J\longrightarrow J_N
 \quad\text{in }C^\infty_{\mathrm{loc}}(N_{\reg}).
\end{equation}
The normalization can be chosen so that
\begin{equation}\label{eq:normalized-shrinker}
 \Ric(g_N)+\nabla^2h=g_N.
\end{equation}
The associated shrinking metric flow, equipped with its normalized conjugate
heat measure, has strictly negative constant Nash entropy.  In particular it
is not the Gaussian shrinking flow on $\C^2$.
\end{proposition}

\begin{proof}
Choose $(x_i,t_i)\in X\times[0,T)$ such that, with $\tau_i=T-t_i$,
\begin{equation}\label{eq:typeii-sequence}
 \tau_i|\Rm(g(t_i))|(x_i)\longrightarrow\infty.
\end{equation}
Smoothness on compact time intervals gives $t_i\to T$. Put
$A_i=t_i/(2\tau_i)\to\infty$ and define
\begin{equation}\label{eq:first-rescaling}
 G_i(u)=\tau_i^{-1}g(t_i+2\tau_i u),
 \qquad -A_i\le u\le0.
\end{equation}
Then
\[
 \partial_uG_i=-2\Ric(G_i),\qquad
 |\Rm(G_i(0))|(x_i)\longrightarrow\infty,\qquad
 R_{G_i}\ge-C\tau_i,
\]
where the last bound follows from the scalar-curvature maximum
principle \cite[Section~2.5 and Lemma~2.8]{BamEntropy}.

\smallskip\noindent
\emph{Uniform entropy bounds and negative density.}
Write $g_0=g(0)$. The fixed-metric Sobolev inequality, obtained
from \cite[Theorem~3.28]{HunterSobolev} in finitely many coordinate
charts, and Jensen's inequality give
\[
 \int_Xv^2\log v^2\,dV_{g_0}
 \le2\log\!\bigl(C(1+E)\bigr),\qquad
 E=\int_X|\nabla v|^2\,dV_{g_0},\qquad \|v\|_{L^2(g_0)}=1.
\]
For $v=(4\pi b)^{-1}e^{-f/2}$ and $b\in[T/4,T/2]$, Perelman's
functional therefore satisfies
\[
 \mathcal W(g_0,f,b)\ge TE-2\log(1+E)-C\ge-Y,\qquad
 \mu(g_0,b)\ge-Y
\]
for a fixed $Y<\infty$.

Set $u_i=t_i/2$ and let
\[
 d\nu_{x_i,u_i;b}=(4\pi b)^{-2}e^{-f_{i,b}}\,
 dV_{\bar g(u_i-b)},\qquad
 \mathcal N_{x_i,u_i}(b)=\int f_{i,b}\,d\nu_{x_i,u_i;b}-2
\]
be the conjugate heat-kernel measure and pointed Nash entropy of
$\bar g(u)=g(2u)$. For large $i$, $u_i\in[T/4,T/2]$.
Entropy monotonicity and the relation between $\mathcal N$ and
$\mathcal W$ \cite[Definition~5.1 and Proposition~5.2]{BamEntropy}
give, for $0<b\le u_i$,
\begin{equation}\label{eq:initial-entropy}
 \begin{aligned}
 \mathcal W_{x_i,u_i}(b)
 &\ge\mathcal W_{x_i,u_i}(u_i)
 \ge\mu(g_0,u_i)\ge-Y,\\
 \mathcal N_{x_i,u_i}(b)
 &=\frac1b\int_0^b\mathcal W_{x_i,u_i}(v)\,dv\ge-Y.
 \end{aligned}
\end{equation}
Thus the pointed entropy of $G_i$ obeys
\[
 \mathcal N_i(a)=\mathcal N_{x_i,u_i}(a\tau_i)\ge-Y
 \qquad(0<a\le A_i).
\]
Bamler's compactness and structure theorems
\cite[Corollary~7.5 and Theorem~7.6]{BamCompact},
\cite[Section~2.1 and Theorems~2.4--2.5]{BamStructure}
give a subsequential ancient metric flow pair
$(\mathcal X,\nu_{x_\infty})$ in the $\mathbb F$ topology,
based at terminal time zero.

Fix $a>0$. Entropy regularity at spatial radius $\sqrt a$
\cite[(9.1), Definition~10.1 and Theorem~10.2]{BamEntropy} gives
\begin{equation}\label{eq:entropy-regularity}
 \mathcal N_i(a)\ge-\varepsilon_0
 \quad\Longrightarrow\quad
 |\Rm(G_i(0))|(x_i)\le\varepsilon_0^{-2}a^{-1}.
\end{equation}
The backward interval $[-a,0]$ is available for large $i$;
this estimate applies at its terminal time. By
\eqref{eq:typeii-sequence}, $\mathcal N_i(a)<-\varepsilon_0$
eventually. Entropy continuity
\cite[Theorem~2.10]{BamStructure} and monotonicity imply
\[
 -Y\le\mathcal N_{x_\infty}(a)\le-\varepsilon_0
 \qquad(a>0),
\]
and hence
\begin{equation}\label{eq:negative-density}
 \Theta:=\lim_{a\downarrow0}\mathcal N_{x_\infty}(a),
 \qquad -Y\le\Theta\le-\varepsilon_0<0.
\end{equation}
The order is $i\to\infty$ at each fixed $a$, followed by $a\downarrow0$.

\smallskip\noindent
\emph{The tangent flow and its realization in the original flow.}
Choose spatial scales $\epsilon_j\downarrow0$ for which the
parabolic rescalings $\mathcal X^{(j)}$ at $x_\infty$ converge
to a tangent pair $\mathcal T$
\cite[Theorem~7.6]{BamCompact},
\cite[Theorem~2.6 and Addendum~2.7]{BamStructure}.
Its Nash entropy is identically $\Theta$
\cite[Theorem~2.11]{BamStructure}.
By \cite[Theorems~2.18 and~2.46]{BamStructure}, its completed
time $-1/2$ slice $(N,g_N)$ is a connected complete Riemannian
orbifold with isolated singularities. The positive conjugate
heat density fixes a single smooth potential $h$ on $N_{\reg}$:
\begin{equation}\label{eq:shrinker-probability}
 d\nu_N=(2\pi)^{-2}e^{-h}\,dV_{g_N},\qquad
 \int_Nd\nu_N=1,\qquad
 \int_Nh\,d\nu_N-2=\Theta.
\end{equation}
The soliton equation at backward time $1/2$ is
$\Ric(g_N)+\nabla^2h=g_N$.

For each fixed $j$, set
\[
 H_{i,j}(v)=\epsilon_j^{-2}G_i(\epsilon_j^2v)
 =\frac{g(t_i+2\epsilon_j^2\tau_i v)}{\epsilon_j^2\tau_i}.
\]
These flows converge to $\mathcal X^{(j)}$ as $i\to\infty$.
The diagonal realization theorem
\cite[Section~6.8 and Theorem~6.58]{BamCompact}
allows $i(j)\nearrow\infty$, chosen after $\epsilon_j$, so that
\[
 H_j:=H_{i(j),j}\xrightarrow{\mathbb F}\mathcal T,\qquad
 \frac{A_{i(j)}}{\epsilon_j^2}>j+1,\qquad
 t_{i(j)}-\epsilon_j^2\tau_{i(j)}>0.
\]
Moreover,
\[
 \mathcal N_{H_j,x_{i(j)},0}(a)
 =\mathcal N_{i(j)}(\epsilon_j^2a)\ge-Y
 \qquad\left(0<a\le\frac{A_{i(j)}}{\epsilon_j^2}\right).
\]
The time-preserving regular-part convergence embeddings
\cite[Theorem~9.31(a)]{BamCompact},
\cite[Theorem~2.5]{BamStructure}, restricted to time $-1/2$,
therefore give embeddings $\psi_j:V_j\to X$ on an exhaustion
of $N_{\reg}$, with $\psi_j^*H_j(-1/2)\to g_N$ smoothly locally.
Set
\begin{equation}\label{eq:diagonal-times}
 s_j=t_{i(j)}-\epsilon_j^2\tau_{i(j)},\qquad
 r_j^2=\epsilon_j^2\tau_{i(j)}.
\end{equation}
Then $H_j(-1/2)=r_j^{-2}g(s_j)$, $s_j\to T$, $r_j\to0$, and
\begin{equation}\label{eq:ratio-calculation}
 \frac{T-s_j}{r_j^2}
 =\frac{\tau_{i(j)}+\epsilon_j^2\tau_{i(j)}}
        {\epsilon_j^2\tau_{i(j)}}
 =1+\epsilon_j^{-2}\longrightarrow\infty.
\end{equation}
This is the two-stage scaling of
\cite[Lemma~2.11, (2.13)--(2.16)]{BCCD}.

\smallskip\noindent
\emph{The K\"ahler structure and extension across the singularities.}
On the same embeddings put
$\widetilde g_j=\psi_j^*(r_j^{-2}g(s_j))$ and
$\widetilde J_j=\psi_j^*J$. In coordinate charts, the parallel
equation is
\[
 \partial_a\widetilde J_j
 =\widetilde J_jC_{j,a}-C_{j,a}\widetilde J_j,
\]
where $C_{j,a}$ are the connection matrices of $\widetilde g_j$.
Orthogonality bounds $\widetilde J_j$; differentiating this equation
and using smooth metric convergence gives bounds of every order.
A diagonal subsequence yields
\[
 \widetilde J_j\longrightarrow J_N,\qquad
 J_N^2=-\mathrm{Id},\quad
 g_N(J_N\cdot,J_N\cdot)=g_N,\quad \nabla^{g_N}J_N=0.
\]
The Nijenhuis tensor vanishes, so the Newlander--Nirenberg theorem
\cite[Theorem~15.4]{CdS} makes $J_N$ integrable. Thus
\[
 \Omega_N=g_N(J_N\cdot,\cdot),\qquad
 \psi_j^*(r_j^{-2}\omega(s_j))\longrightarrow\Omega_N
 \quad\text{in }C^\infty_{\mathrm{loc}}(N_{\reg})
\]
defines a K\"ahler structure on $N_{\reg}$.

In a singularity chart $B/\Gamma$, lift the smooth metric to
$\widehat g$ on $B\subset\R^4$ and write
$S=\widehat g-\Ric(\widehat g)$, $w_b=\partial_b\widehat h$.
On $B\setminus\{0\}$ the lifted tensors satisfy
\[
 \partial_a\widehat J=\widehat JC_a-C_a\widehat J,\qquad
 \partial_a\widehat h=w_a,\qquad
 \partial_aw_b=S_{ab}+\Gamma^c_{ab}w_c.
\]
On a smaller ball, join a fixed point $p_*$ to each punctured
point by radial and spherical paths of uniformly bounded length $L$.
Parallel transport and $\nabla^2\widehat h=S$ give
\[
 |\nabla\widehat h|(x)
 \le|\nabla\widehat h|(p_*)+L\|S\|_{C^0}=:G,\qquad
 |\widehat h(x)|\le|\widehat h(p_*)|+LG.
\]
Together with orthogonality of $\widehat J$, the first-order
system bounds $U=(\widehat J,\widehat h,w)$ and its first derivatives.
Points in $B_\eta\setminus\{0\}$ can be joined there by radial
and spherical paths of length at most $(2+\pi)\eta$. Hence
\[
 \operatorname{osc}_{B_\eta\setminus\{0\}}U\le C\eta,
\]
so $U$ has a unique continuous value at zero.
Write the system as $\partial_aU=F_a(x,U)$, with smooth coefficients
and affine dependence on $U$.
The radial equations
\[
 \frac{d}{dt}V_x(t)=\sum_a x^aF_a(tx,V_x(t)),\qquad
 V_x(0)=U(0),\qquad U(x)=V_x(1)
\]
and smooth dependence for ordinary differential equations
\cite[Chapter~9]{Lee} give a smooth extension.
Uniqueness preserves $\Gamma$-invariance and agreement on overlaps.
Thus $J_N$ and $h$ extend as smooth orbifold tensors, satisfying
\eqref{eq:normalized-shrinker} everywhere.

The isotropy representation commutes with $\widehat J(0)$ and
is orthogonal, so it lies in $U(2)$; isolated isotropy implies
freeness on $S^3$. To obtain holomorphic linear quotient coordinates,
choose centred holomorphic coordinates $z$ with $dz_0=\mathrm{Id}$ and set
\[
 F(p)=\frac1{|\Gamma|}\sum_{\gamma\in\Gamma}
       (D\gamma_0)^{-1}z(\gamma p),\qquad
 dF_0=\mathrm{Id},\qquad
 F(\delta p)=D\delta_0F(p).
\]
The holomorphic inverse function theorem gives the required chart.

Finally, $\Ric(g_N)$ and $g_N$ are $J_N$-invariant. With
$A=(\nabla^2h)^\sharp$, the soliton equation therefore yields
\[
 AJ_N=J_NA,\qquad
 \mathcal L_{\nabla h}J_N=J_NA-AJ_N=0.
\]
Thus the gradient is real holomorphic. The singular set has zero
measure, so the extension preserves \eqref{eq:shrinker-probability}
and the constant entropy $\Theta<0$. The Gaussian shrinking flow
has Nash entropy zero and is therefore excluded.
\end{proof}

\begin{proof}[Proof of Theorem~\ref{thm:main}]
If $\alpha=0$, then $[\omega_0]=2\pi Tc_1(X)$ and $X$ is a del Pezzo
surface. The soliton existence results \cite{TianDelPezzo,WZ} and
normalized-flow convergence \cite[Corollaries~1.5 and~4.3]{DS} give
\[
 \widetilde g(s)=\frac{g(t)}{T-t},\qquad
 s=\log\frac{T}{T-t},\qquad
 \sup_{X\times[0,\infty)}|\Rm(\widetilde g(s))|<\infty.
\]
Since $(T-t)|\Rm(g(t))|=|\Rm(\widetilde g(s))|$, this proves
\eqref{eq:typeI} in the extinction case.
Assume henceforth $\alpha\ne0$, so that Proposition~\ref{prop:reduction}
and the results of Section~\ref{sec:limits} apply.
Suppose that \eqref{eq:typeI} fails.
Proposition~\ref{prop:twoscale} gives $(N,g_N,J_N,h)$ and sequences
$s_j,r_j$ satisfying \eqref{eq:twoscale-ratio}--\eqref{eq:twoscale-convergence}.

\smallskip\noindent
\emph{Exclusion of orbifold points.}
Suppose that $p\in N$ is singular. In equivariant normal coordinates,
its chart is $B_R(0)/\Gamma$, where $\Gamma\subset U(2)$ is nontrivial
and acts freely on $S^3$, and the lifted metric satisfies
$\widetilde g_{ab}(0)=\delta_{ab}$. Set
\[
 A_\Gamma^+=\{1/3<|z|<3\}/\Gamma,\qquad
 A_\Gamma=\{1/2<|z|<2\}/\Gamma,\qquad K=\overline A_\Gamma.
\]
Choose $\eta_k\downarrow0$ with $3\eta_k<R$. In this chart the maps
$\iota_k:A_\Gamma^+\to N_{\reg}$, $[z]\mapsto[\eta_kz]$, satisfy
\[
 \eta_k^{-2}\iota_k^*g_N
 =\sum_{a,b}\widetilde g_{ab}(\eta_kz)\,dz^a\,dz^b
 \longrightarrow g_E
 \quad\text{in }C^\infty(K).
\]
Here the coefficient identity is written on the covering annulus.
Put $g_j^\#=\psi_j^*(r_j^{-2}g(s_j))$. For each fixed $k$,
\eqref{eq:twoscale-convergence} allows a choice $j(k)>j(k-1)$ such that
\[
 \iota_k(K)\subset V_{j(k)},\qquad
 E_k:=\eta_k^{-2}\iota_k^*(g_{j(k)}^\#-g_N),\qquad
 \|E_k\|_{C^k(K,g_E)}<\frac1k.
\]
Since $V_{j(k)}$ is open, the composition below is defined on a
neighbourhood of $K$. Set
\[
 \rho_k=\eta_kr_{j(k)},\qquad
 \Phi_k=(\psi_{j(k)}\circ\iota_k)|_{A_\Gamma}.
\]
Then $\rho_k\to0$, $s_{j(k)}\to T$, and
\[
 \Phi_k^*(\rho_k^{-2}g(s_{j(k)}))-g_E
 =E_k+\bigl(\eta_k^{-2}\iota_k^*g_N-g_E\bigr)
 \longrightarrow0
 \quad\text{in }C^\infty_{\mathrm{loc}}(A_\Gamma).
\]
This contradicts Lemma~\ref{lem:noannulus}. Thus $N$ is a complete
smooth gradient K\"ahler shrinker.

\smallskip\noindent
\emph{Exactness and classification.}
By \eqref{eq:twoscale-ratio}--\eqref{eq:twoscale-convergence},
Proposition~\ref{prop:exact} gives
$\Omega_N=d\lambda$ and excludes compact complex curves in $N$.
Moreover, $N$ is noncompact, since otherwise
\[
 0<\int_N\Omega_N^2
 =\int_Nd(\lambda\wedge\Omega_N)=0.
\]
To match the soliton scale in \cite[Section~1]{LW}, put
\[
 \begin{aligned}
 \widehat g&=2g_N,\qquad
 \widehat\Omega=2\Omega_N=d(2\lambda),\\
 \Ric(\widehat g)+\nabla_{\widehat g}^2h
 &=g_N=\tfrac12\widehat g,\qquad
 \nabla_{\widehat g}h=\tfrac12\nabla_{g_N}h.
 \end{aligned}
\]
Completeness and the real holomorphic gradient are preserved.
An additive constant in $\widehat h=h+c$ gives the potential
normalization $R_{\widehat g}+|\nabla\widehat h|_{\widehat g}^2
=\widehat h$ of \cite[Section~1, (1.1)--(1.2)]{LW}.
Thus \cite[Theorem~1.1]{LW}, which imposes no bounded-curvature
hypothesis, gives a biholomorphic isometry to one of the four
noncompact models:
\[
 \C^2,\qquad \Bl_0\C^2,\qquad
 \PP^1\times\C,\qquad \Bl_p(\PP^1\times\C).
\]
The last three contain compact holomorphic rational curves: the
exceptional curve in either blow-up, and $\PP^1\times\{z\}$ in the
product. They are therefore excluded by Proposition~\ref{prop:exact}.
The remaining model is the Gaussian shrinker, excluded by the
strictly negative entropy in Proposition~\ref{prop:twoscale}.
This contradiction proves \eqref{eq:typeI}.
\end{proof}


\bigskip
\noindent\textsc{Yeyun Xu}\\
The Fundamental Science Research Center for Geometry and Algebra,\\
School of Mathematical Sciences, East China Normal University,
Shanghai, China\\
\textit{E-mail address:}
\href{mailto:1322706807@qq.com}{\texttt{1322706807@qq.com}}

\medskip
\noindent\textsc{Linfeng Zhou}\\
The Fundamental Science Research Center for Geometry and Algebra,\\
School of Mathematical Sciences, East China Normal University,
Shanghai, China\\
\textit{E-mail address:}
\href{mailto:lfzhou@math.ecnu.edu.cn}{\texttt{lfzhou@math.ecnu.edu.cn}}

\begin{thebibliography}{99}

\bibitem{Markov}
N.~Adaloglou, J.~Brendel, J.~Evans, J.~Hauber, and F.~Schlenk,
\emph{Markov staircases} (2025).
\href{https://arxiv.org/abs/2509.03224v2}{arXiv:2509.03224v2}.

\bibitem{Auroux}
D.~Auroux,
\href{https://ocw.mit.edu/courses/18-966-geometry-of-manifolds-spring-2007/8de7efaaa019aef8237177a82020050c_lect12.pdf}{\emph{Symplectic Geometry}}, MIT 18.966 course notes, Spring 2007,
Lecture~12.

\bibitem{BamEntropy}
R.~H.~Bamler,
\emph{Entropy and heat kernel bounds on a Ricci flow background} (2021).
\href{https://arxiv.org/abs/2008.07093v3}{arXiv:2008.07093v3}.

\bibitem{BamCompact}
R.~H.~Bamler,
\emph{Compactness theory of the space of super Ricci flows},
Invent. Math. \textbf{233} (2023), no.~3, 1121--1277.
\href{https://arxiv.org/abs/2008.09298v2}{arXiv:2008.09298v2}.

\bibitem{BamStructure}
R.~H.~Bamler,
\emph{Structure theory of non-collapsed limits of Ricci flows} (2021).
\href{https://arxiv.org/abs/2009.03243v2}{arXiv:2009.03243v2}.

\bibitem{BCCD}
R.~H.~Bamler, C.~Cifarelli, R.~J.~Conlon, and A.~Deruelle,
\emph{A new complete two-dimensional shrinking gradient
K\"ahler--Ricci soliton},
Geom. Funct. Anal. \textbf{34} (2024), no.~2, 377--392.
\href{https://arxiv.org/abs/2206.10785v1}{arXiv:2206.10785v1}.

\bibitem{BottTu}
R.~Bott and L.~W.~Tu,
\emph{Differential Forms in Algebraic Topology},
Graduate Texts in Mathematics, Vol.~82, Springer, 1982.
\href{https://doi.org/10.1007/978-1-4757-3951-0}{doi:10.1007/978-1-4757-3951-0}.

\bibitem{CdS}
A.~Cannas da Silva,
\href{https://people.math.ethz.ch/~acannas/Papers/lsg.pdf}{\emph{Lectures on Symplectic Geometry}},
Lecture Notes in Mathematics, Vol.~1764, Springer, 2001 (2006 version).

\bibitem{CNPP}
C.~Caubel, A.~N\'emethi, and P.~Popescu-Pampu,
\emph{Milnor open books and Milnor fillable contact 3-manifolds},
Topology \textbf{45} (2006), 673--689.
\href{https://arxiv.org/abs/math/0409160v2}{arXiv:math/0409160v2}.

\bibitem{CP}
H.~Choi and J.~Park,
\emph{A Lefschetz fibration on minimal symplectic fillings of a quotient
surface singularity} (2019).
\href{https://arxiv.org/abs/1802.03304v2}{arXiv:1802.03304v2}.

\bibitem{CCHZ}
C.~Cifarelli, R.~J.~Conlon, M.~Hallgren, and J.~Zhang,
\emph{Finite time singularities of the Ricci flow on compact
K\"ahler surfaces are of Type I} (2026).
\href{https://doi.org/10.48550/arXiv.2609.16733}{arXiv:2609.16733}.

\bibitem{CHM}
R.~J.~Conlon, M.~Hallgren, and Z.~Ma,
\emph{Non-collapsed finite time singularities of the Ricci flow on compact
K\"ahler surfaces are of Type I} (2025).
\href{https://arxiv.org/abs/2502.19804v2}{arXiv:2502.19804v2}.

\bibitem{Demailly}
J.-P.~Demailly,
\href{https://www-fourier.univ-grenoble-alpes.fr/~demailly/manuscripts/agbook.pdf}{\emph{Complex Analytic and Differential Geometry}},
2012.

\bibitem{DS}
R.~Dervan and G.~Sz\'ekelyhidi,
\emph{The K\"ahler--Ricci flow and optimal degenerations} (2018).
\href{https://arxiv.org/abs/1612.07299v5}{arXiv:1612.07299v5}.

\bibitem{FG}
O.~Fujino and Y.~Gongyo,
\emph{Log pluricanonical representations and abundance conjecture} (2012).
\href{https://arxiv.org/abs/1104.0361v3}{arXiv:1104.0361v3}.

\bibitem{Hatcher}
A.~Hatcher,
\href{https://pi.math.cornell.edu/~hatcher/AT/AT.pdf}{\emph{Algebraic Topology}},
Cambridge University Press, Cambridge, 2002.

\bibitem{HunterSobolev}
J.~K.~Hunter,
\href{https://www.math.ucdavis.edu/~hunter/pdes/pde_notes.pdf}{\emph{Notes on Partial Differential Equations}},
University of California, Davis, 2014.

\bibitem{Lee}
J.~M.~Lee,
\emph{Introduction to Smooth Manifolds}, second edition,
Graduate Texts in Mathematics, Vol.~218, Springer, 2013.
\href{https://doi.org/10.1007/978-1-4419-9982-5}{doi:10.1007/978-1-4419-9982-5}.

\bibitem{LZnef}
T.-J.~Li and W.~Zhang,
\emph{$J$-holomorphic curves in a nef class} (2015).
\href{https://arxiv.org/abs/1210.3337v3}{arXiv:1210.3337v3}.

\bibitem{LW}
Y.~Li and B.~Wang,
\emph{On K\"ahler Ricci shrinker surfaces} (2025).
\href{https://arxiv.org/abs/2301.09784v2}{arXiv:2301.09784v2}.

\bibitem{Lipman}
J.~Lipman,
\href{https://www.numdam.org/item/PMIHES_1969__36__195_0.pdf}{\emph{Rational singularities with applications to algebraic surfaces
and unique factorization}},
Publ. Math. Inst. Hautes \'{E}tudes Sci. \textbf{36} (1969), 195--279.

\bibitem{MT}
M.~Miao and G.~Tian,
\emph{A note on K\"ahler--Ricci flow on Fano threefolds},
Peking Math. J. \textbf{8} (2025), no.~1, 191--199.
\href{https://doi.org/10.1007/s42543-023-00078-0}{doi:10.1007/s42543-023-00078-0}.

\bibitem{Stacks}
The Stacks Project Authors,
\emph{The Stacks Project}.
\href{https://stacks.math.columbia.edu/tag/0B8Y}{Tag~0B8Y},
\href{https://stacks.math.columbia.edu/tag/0B3P}{Tag~0B3P},
and \href{https://stacks.math.columbia.edu/tag/02SE}{Tag~02SE}.

\bibitem{TianDelPezzo}
G.~Tian,
\emph{On Calabi's conjecture for complex surfaces with positive first
Chern class}, Invent. Math. \textbf{101} (1990), 101--172.
\href{https://doi.org/10.1007/BF01231499}{doi:10.1007/BF01231499}.

\bibitem{TZcollapse}
V.~Tosatti and Y.~Zhang,
\emph{Finite time collapsing of the K\"ahler--Ricci flow on threefolds},
Ann. Sc. Norm. Super. Pisa Cl. Sci. (5) \textbf{18} (2018), no.~1, 105--118.
\href{https://arxiv.org/abs/1507.08397v2}{arXiv:1507.08397v2}.

\bibitem{WZ}
X.-J.~Wang and X.~Zhu,
\emph{K\"ahler--Ricci solitons on toric manifolds with positive first
Chern class}, Adv. Math. \textbf{188} (2004), no.~1, 87--103.
\href{https://doi.org/10.1016/j.aim.2003.09.009}{doi:10.1016/j.aim.2003.09.009}.

\bibitem{WendlLectures}
C.~Wendl,
\emph{Lectures on holomorphic curves in symplectic and contact geometry} (2014).
\href{https://arxiv.org/abs/1011.1690v2}{arXiv:1011.1690v2}.

\bibitem{XZ}
T.~Xu and Z.~Zhang,
\emph{Finite Time Singularities of Collapsing K\"ahler Ricci Flow on
Ruled Surfaces} (2026).
\href{https://arxiv.org/abs/2609.01442v1}{arXiv:2609.01442v1}.

\bibitem{XZTypeI}
T.~Xu and Z.~Zhang,
\emph{Collapsed Finite Time Singularities of the K\"ahler-Ricci Flow
on Complex Surfaces are of Type I} (2026).
\href{https://doi.org/10.48550/arXiv.2609.18834}{arXiv:2609.18834}.

\end{thebibliography}
\end{document}